\documentclass{articlenotescls}
\usepackage{gencatmacros}

\setboolean{cdbaselineon}{true}
\setboolean{showcdpunctuation}{true}

\DeclareMathOperator{\el}{\mathbf{el}}
\DeclareMathOperator*{\Coprod}{\displaystyle{\coprod}}

\newcommand{\Sk}{\mathrm{Sk}}
\newcommand{\Cosk}{\mathrm{Cosk}}
\newcommand{\sk}{\mathrm{sk}}
\newcommand{\cosk}{\mathrm{cosk}}
\newcommand{\nd}{\mathrm{nd}}

\newcommand{\Lan}{\mathrm{Lan}}
\newcommand{\Ran}{\mathrm{Ran}}
\newcommand{\Chain}{\mathrm{chain}}

\newcommand{\iso}{\mathrm{iso}}

\newcommand{\coeq}{\mathrm{coeq}}

\newcommand{\iep}{\mathrm{IEP}}

\newcommand{\Graph}{\mathbf{Graph}}
\newcommand{\MarCat}{\mathbf{MarCat}}
\newcommand{\RelCat}{\mathbf{RelCat}}

\newcommand{\cL}{\mathcal{L}}

\begin{document}
	\pagestyle{empty}
	\pagenumbering{Roman}
	
	\title{Revisiting colimits in \tpdf{$\Cat$} and homotopy category}
	\author{Varinderjit Mann}
	\date{\today}
	\maketitle
	
	\begin{abstract}
		In this paper, we justify and make precise an elementary approach that establishes the existence of (co)limits in $\Cat$. This approach, while conceptually evident, has not been made fully explicit or systematically described in the literature. We first demonstrate an equivalence between the existence of the homotopy category functor $h : \sSet \rightarrow \Cat$ and the existence of a specific class of weighted colimits in $\Cat$. We then construct these weighted colimits explicitly by using certain properties of simplicial sets and the nerve functor. Consequentially, the embedding $N : \Cat \hookrightarrow \sSet$ is reflective, and can be used to infer the (co)completeness of $\Cat$. Finally, we use this approach to reformulate the construction of coequalizers and localizations in $\Cat$.
	\end{abstract}
	
	\tableofcontents
	\clearpage 
	
	\pagestyle{scrheadings}
	\pagenumbering{arabic}
	
	\section{Introduction}
	\subsection{Motivation}

It is a well-known result that the category of all (small) categories $\Cat$ is cocomplete \cite{betti1983,borceux1,bednar1999,reihlcat,wolff1974}. In \cite{betti1983} and \cite{wolff1974}, cocompleteness follows from their main results, which imply that the canonical forgetful functor $U : \Cat \rightarrow \Graph$ is \emph{finitary monadic}. However, these papers work with the general case of $\cV$\emph{-enriched categories} and $\cV$\emph{-graphs}. And, this generality makes the results comparatively inefficient when applied to the case of $\Cat$ because the category $\Set$ possesses rich additional structures.

On the other hand, \cite[Prop.~4.1]{bednar1999} provides a construction of coequalizers in $\Cat$ by utilizing the theory of \emph{generalized congruences} and a classification of \emph{epimorphisms}. This approach is the most elementary in spirit, but the constructions themselves are fairly intricate. A similar situation occurs with another explicit construction of coequalizers, as given in \cite[Prop.~5.1.7]{borceux1}. However, a clearer and more recent exposition using coequalizers can be found in \cite[Thm.~1.4.7]{yauinvcat2020}.

In this paper, we describe a different elementary perspective on the cocompleteness of $\Cat$ in terms of a \emph{reflective embedding} of $\Cat$ into $\sSet$ arising from the canonical inclusion $\bDelta \hookrightarrow \Cat$. This embedding $N : \Cat \rightarrow \sSet$, called the \emph{nerve functor}, is well known and commonly taken to be reflective in the literature. However, the property of $N$ being reflective is dependent upon the cocompleteness of $\Cat$. Thus, assuming this property leads to either an inadequate or a circular argument\footnote{To our knowledge, similar arguments seem to have occurred in books such as \cite{reihlcat}.}. \emph{Quite surprisingly}, we were unable to find a direct and self-contained proof in the literature.

We tackle this problem by demonstrating the existence of a certain restricted class of colimits in $\Cat$, which notably circumvents the need to construct all coequalizers as done in \cite[Prop.~4.1]{bednar1999} and \cite[Prop.~5.1.7]{borceux1}. \textit{Weighted colimits} are used as a central tool to express and study this class of colimits. In general, many concepts such as \emph{pointwise Kan extensions} or even usual colimits are more naturally described using the language of weighted colimits. Although weighted colimits are defined generally over an arbitrary base of enrichment $\cV$ (a \emph{symmetric monoidal closed category}), their description becomes considerably simplified over the base category $\Set$. Consequentially, weighted colimits over $\Set$ are both conceptually natural and computationally tractable.

A potential caveat to our approach is that there is no immediate generalization to the enriched case as in \cite{betti1983,wolff1974}. This is because $\cV\text{-}\Cat$ doesn't have any canonical well behaved simplex category associated with it, unlike the canonical inclusion $\bDelta \hookrightarrow \Cat$. Evidently, this highlights the importance of the additional structure on the category $\Set$ that this approach relies upon. Another possible objection is that the core idea of  this approach resembles the definition of generalized congruences as in \cite{bednar1999} or the construction of coequalizers as in \cite{borceux1}. However, we specifically circumvent the need to construct coequalizers or congruences in full generality. Instead, the restricted class of colimits is not only easier to construct and understand, but also closely related to the well-understood pointwise colimits in simplicial sets. 

Finally, we reemphasize that the \emph{primary motivation} for this paper is to provide a direct and complete account of this seemingly natural approach, which appears to be absent from the literature. A secondary motivation is to offer a more accessible proof of this important result.

\subsection{Main results and implications}

Our first goal is to reduce the cocompleteness of $\Cat$ to a simpler requirement. Thus, we first provide the following proposition (as seen in \cref{sec:handN}), as a key step toward the cococompleteness.

\begin{proposition*}
	The category $\Cat$ is cocomplete iff for every simplicial set $X \in \sSet$ a certain weighted colimit denoted $X \star [-]$ exists. 
\end{proposition*}

This proposition is primarily a consequence of the characterization of pointwise left Kan extension using weighted colimits. Although the weighted colimits $X \star [-]$ themselves are not very straightforward to compute, they are cocontinuous in the weights $X \in \sSet$.  This fact is then used in conjunction with the skeletal filtration of $X$ to prove the following theorem.

\begin{theorem*}
	For all $X \in \sSet$, the weighted colimits $X \star [-]$ exist.
\end{theorem*}

Thereafter, we recover the well established pair of adjoints $h \dashv N$ between $\sSet$ and $\Cat$. Combined with another standard result regarding reflective embeddings\cite[Prop.~4.5.15]{reihlcat}, this ensures that $\Cat$ is both complete and cocomplete (see \cref{sec:ladj}).

\begin{theorem*}
	The category $\Cat$ is both complete and cocomplete.
\end{theorem*}

For an immediate implication, we ask what other relevant adjunctions are derivable directly from the $h \dashv N$ adjunction. As seen in \cref{sec:deradj}, one such example is a quadruple of adjoints between $\Cat$ and $\Set$, derived from a similar such quadruple between $\sSet$ and $\Set$. Another relevant example turns out to be the usual free category adjunctions on graphs or reflexive graphs. Thereafter, \cref{lemma:coeqcat} provides an explicit description of the coequalizers in $\Cat$. Finally, we reinterpret localizations in $\Cat$ using the $h \dashv N$ adjunction to describe the desired defining colimits.

\subsection{Overview}

In Section 2, we briefly state the relevant well-known results on the Grothendieck construction, left Kan extensions, weighted colimits, and the nerve--realization pair. We only include them for accessibility and ease of reference. An advanced reader may safely glance over this section apart from the \cref{lemma:wlimislim}.

In Section 3, we specialize to the nerve--realization pair for simplicial sets and categories. In particular, this section studies only the existence of certain weighted colimits and their interaction with the $2$-skeleton functor on simplicial sets. The key point of this section is to ensure that no circular reasoning arises from the properties used in our main proof.

Section 4 states and proves the main results that establish the existence of the desired nerve--realization pair. An advanced reader may focus solely on this section if they are confident that all the tools used are free from any circular assumptions. However, this point is not entirely trivial.

Section 5 first discusses other relevant adjunctions between $\Cat$ and $\Set$ derived from the nerve--realization pair using the $0,1$-skeleton adjunction on $\sSet$.Then, it derives an explicit description of the coequalizers in $\Cat$. Finally, it reformulates the localizations construction in the category $\Cat$ using colimits.

\subsection{Notation and basic results}

\begin{enumerate}[leftmargin=*,labelsep=0.5em]
	\item The category $\Set$ is the category whose objects are all sets, and whose morphisms are functions between them. The category $\Set$ is complete and cocomplete.
	
	\item The category $\Cat$ is the category whose objects are all \emph{small} categories, and whose morphisms are functors between them.
	
	\item The simplex category $\bDelta$ is the full subcategory of $\Cat$ on the categories which are the linearly ordered sets $[n]=\{0 \rightarrow 1 \rightarrow \cdots \rightarrow n \}$ for all $n\geq 0$.
	
	\item The category $\sSet \vcentcolon= \Set^{\bDelta^{\op}}$ is the category of simplicial sets. It is complete and cocomplete, with limits and colimits computed pointwise. The same is true of $\Set^{\cC^{\op}}$ for any small category $\cC$. 
	
	\item The cartesian product $\cC \times \cD$ of two categories is defined by: $\obj(\cC \times \cD) = \{(c,d) \mid c \in \cC, d \in \cD\}$, and $\mor(\cC \times \cD)=\{(f , g) \mid f \in \mor \cC, g \in \mor \cD\}$. The product $\cC \times \cD$ then forms a category with componentwise composition and identities. Furthermore, it satisfies the universal property of a categorical product and extends to a functor $- \times - : \Cat \times \Cat \rightarrow \Cat$.
	
	\item The Yoneda embedding is denoted as the functor \footnote{$\Yo$ is the Japanese character for \enquote{yo} in Hiragana syllabary. This symbol also looks like \enquote{Y} in the English alphabet.} $\Yo: \cC \rightarrow \Set^{\cC^{\op}}$, which is given objectwise by $c \mapsto \cC(-,c)$. And, it satisfies the following natural isomorphism defined by $\alpha \mapsto \alpha_{c}(1_c)$:
	\[\begin{tikzcd}[column sep=scriptsize]
		&&& {\left(\Set^{\cC^{\op}}\right)^{\op} \times \Set^{\cC^{\op}}} \\
		{(c,X)} &&&&&& {\Set^{\cC^{\op}}(\Yo(c),X)} \\
		& {\cC^{\op} \times \Set^{\cC^{\op}}} &&&& \Set \\
		&&&&&& {X(c)}
		\arrow["{\Set^{\cC^{\op}}(-,-)}", from=1-4, to=3-6]
		\arrow["\in"{description}, dashed, from=2-1, to=3-2]
		\arrow["\in"{description}, dashed, from=2-7, to=3-6]
		\arrow["{\alpha \mapsto \alpha_c(1_c)}", from=2-7, to=4-7]
		\arrow["{\Yo^{\op} \times 1}", from=3-2, to=1-4]
		\arrow[""{name=0, anchor=center, inner sep=0}, "{\mathrm{eval}(-,-)}"', from=3-2, to=3-6]
		\arrow["\in"{description}, dashed, from=4-7, to=3-6]
		\arrow["\cong"', between={0}{0.8}, Rightarrow, from=1-4, to=0]
	\end{tikzcd}\]
	Here, the functor $\mathrm{eval}(-,-)$ is defined by sending $(c,X) \mapsto X(c)$. In particular, we have the isomorphism $\Set^{\cC^{\op}}(\Yo(c),X)\cong X(c)$ \footnote{We can use suitable enlargement of the universe to define $\Set^{\cC^{\op}}$, if necessary. However, the natural isomorphism itself holds true with or without this enlargement.},  which is natural in both $c \in \cC$ and $X \in \Set^{\cC^{\op}}$. Moreover, $\Yo$ is fully faithful \cite[p.~59--60]{reihlcat}.

\end{enumerate}

	\section{Preliminaries}
	The primary goal of this section is to serve as a convenient reference. Most of these definitions and results are well known, but we hope that our presentation still provides valuable insights and increases the accessibility of this paper.

\subsection{Category of elements}

Any function $f: X \rightarrow Y \in \Set$ can be transformed into a discrete functor $\bar{f} : Y \rightarrow \Set$ by considering fibers over each point $y \in Y$. This transformation is completely reversible by bundling the codomain of $\bar{f}$ using coproducts. In particular, there is a bijection between functions $f : X \rightarrow Y \in \Set$ with fixed $Y \in \Set$ and discrete functors $\bar{f} : Y \rightarrow \Set$. A natural generalization is to ask whether arbitrary functors $F : \cC \rightarrow \Cat$ or $F: \cC^{\op} \rightarrow \Cat$ can be bundled in some way. The \emph{Grothendieck construction} answers this question in the positive. Our requirements only include functors of the form $F : \cC \rightarrow \Set$ or $F : \cC^{\op} \rightarrow \Set$. In this situation, the Grothendieck construction is also known as the \emph{Category of elements} construction.

\begin{definition}[Category of elements]\label{defn:grconst} 
	\begin{enumerate}[leftmargin=*,labelsep=.5em,itemsep=0.5em,topsep=0em, parsep=0.25em]
		\item[]
		\item For all $F : \cC \rightarrow \Set$, the (covariant) category of elements $\el F$ is a category defined by:
		\begin{description}
			\item[a) $\obj(\el F)$] the collection $\{(*\xrightarrow{x} F(c), c) : c \in \cC, x \in F(c)\}$,  
			\item[b) $\mor(\el F)$] any morphism $(* \xrightarrow{x} F(c),c) \rightarrow (*\xrightarrow{y} F(d),d) \in \el F$ is given by a morphism $f: c \rightarrow d \in \cC$ that satisfy $F(f)(x)=y$.
		\end{description}
		
		\item For all $F : \cC^{\op} \rightarrow \Set$, the (contravariant) category of elements $\el F$ is a category defined by:
		\begin{description}
			\item[a) $\obj(\el F)$] the collection $\{(*\xrightarrow{x} F(c), c) : c \in \cC, x \in F(c)\}$,
			\item[b) $\mor(\el F)$] any morphism $(* \xrightarrow{x} F(c),c) \rightarrow (*\xrightarrow{y} F(d),d) \in \el F$ is given by a morphism $f: c \rightarrow d \in \cC$ that satisfy $F(f)(y)=x$.
		\end{description}
	\end{enumerate}
	Both of these $\el F$ categories come with a canonical projection map $\el F \rightarrow \cC$. 
\end{definition}

A routine check shows that both the constructions $\el$ are functorial in the respective functor category.

\begin{lemma}[{\cite[Lemma.~10.3.5]{yau2cat2021}}] \label{lemma:grfunct}
	The covariant Grothendieck construction $\el$ defines a functor $\el : \Set^{\cC} \rightarrow \Cat/\cC$. The action on transformations $\alpha : F \Rightarrow G$ is given by defining $\el \alpha : \el F \rightarrow \el G$ elementwise using $\el \alpha (c,x) = (c,\alpha_{c}(x))$. A similar lemma holds for contravariant Grothendieck construction $\el : \Set^{\cC^{\op}} \rightarrow \Cat/\cC$.
\end{lemma}

\begin{example}[representable functors]
	The category of elements of the representable functor $\cC(-,c)$ is $\cC/c$, whereas, for the corepresentable functor $\cC(c,-)$, it is $c/\cC$. Notice that these categories have the terminal and initial object respectively. In a way, the existence of these objects is equivalent to the assertion of the Yoneda Lemma. 
\end{example}

Another important use of the category of elements construction is for simplicial sets. Any simplicial set $X$ is a functor $X : \bDelta^{\op} \rightarrow \Set$. Hence, we can compute it's category of elements $\el X$. This category is usually called the \emph{category of simplices} of $X$.

\begin{example}[category of simplices]
	For any simplicial set $X$, the category of elements $\el X$ has:
	\begin{description}
		\item[$\obj(\el X)$] - all simplicial maps $\Delta^n \rightarrow X$ for $n\geq 0$,
		\item[$\mor(\el X)$] - all commuting triangles 
		
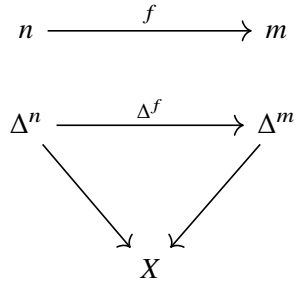
\[\begin{tikzcd}
			n && m \\
			{\Delta^n} && {\Delta^m} \\
			\\
			& X
			\arrow["f", from=1-1, to=1-3]
			\arrow["{\Delta^f}", from=2-1, to=2-3]
			\arrow[from=2-1, to=4-2]
			\arrow[from=2-3, to=4-2]
		\end{tikzcd} \cdpunct[8pt]{.}\]
	\end{description} 
\end{example}
 
\subsection{Weighted limits}

Weighted limits and colimits generalize ordinary limits and colimits in a canonical way and provide simple formulas for plenty of constructions in category theory. In our case, they provide relevant formulas for computing \emph{Kan extensions}. We mainly present the theory of weighted colimits and derive the theory of weighted limits using duality.

\begin{definition}[Weighted colimit]\label{defn:wcolim}
	Let $W : \cC^{\op} \rightarrow \Set$ and $F :\cC \rightarrow \cD$ be two arbitrary functors. The weighted colimit $W \star F$, if it exists, is defined by the following natural isomorphism in $d \in \cD$:
	\[\cD(W \star F, d) \cong \Set^{\cC^{\op}}(W, \cD(F(-),d)) \cdpunct{.}\]
	The functor $W$ is called the \emph{weight functor} for the weighted colimit $W \star F$.
\end{definition}

\begin{remark}
	The natural isomorphism is itself a part of the data that defines the weighted colimit $W \star F$. In particular, the image of the map $1_{W \star F} : W \star F \rightarrow W \star F$, say $\lambda_{W \star F} \in \Set^{\cC^{\op}}(W,\cD(F(-),W\star F))$ is also necessarily a part of this data. In fact, the Yoneda lemma implies that a weighted colimit is completely specified by the pair $(W \star F, \lambda_{W \star F})$, which together are a \emph{universal object} $W \star F$ and a \emph{universal weighted cocone} $\lambda_{W \star F}$. 
\end{remark}

This definition of weighted colimit is given in terms of a natural isomorphism. Hence, whenever all the weighted colimits exist, they extend to a $2$-ary functor $- \star - : \Set^{\cC^{\op}} \times \cD^{\cC} \rightarrow \cD$ using the Yoneda lemma. This $2$-ary functor is also cocontinuous in the weights argument (the left input). This again follows from the defining natural isomorphism. 

\begin{lemma}
	Whenever all the desired weighted colimits exist, there is a canonical way to assemble them into a $2$-ary functor $- \star - : \Set^{\cC^{\op}} \times \cD^{\cC} \rightarrow \cD$, which is cocontinuous in its weights argument. That is, the canonical comparison map $\Colim\limits_{i} (W_i \star F) \rightarrow (\Colim\limits_{i} W_i) \star F$ is an isomorphism. 
\end{lemma}

The Yoneda embedding acts as a unit for the functor $- \star - $ in two distinct ways which provide an important class of examples for weighted colimits. 

The first form of unitality is where $\Yo$, as a whole, acts as the right unit and is demonstrated by the following lemma.

\begin{lemma}
	For every functor $W : \cC^{\op} \rightarrow \Set$, there exists a natural isomorphism $W \star \Yo \cong W$.  
\end{lemma}

\begin{proof}
	The Yoneda lemma implies the existence of the following isomorphism natural in $W$ and $X$:
	\[\begin{tikzcd}
		{\Set^{\cC^{\op}}(W,\Set^{\cC^{\op}}(\Yo,X))} && {\Set^{\cC^{\op}}(W,X)}
		\arrow["\cong", from=1-1, to=1-3]
	\end{tikzcd}\cdpunct{.}\]
	This implies that $W$ satisfies the universal property of $W \star \Yo$ as desired.
\end{proof}

The second unitality is symmetric and uses the representable functors $\Yo(c)= \cC(-,c) : \cC^{\op} \rightarrow \Set$, taken once as a weight functor and once a functor whose weighted colimit is being computed. This is again directly verified using the Yoneda isomorphism and the cartesian closure of $\Set$.

\begin{lemma}[Left unit]
	For every functor $F: \cC \rightarrow \cD$, there exists an isomorphism $\Yo(c) \star F \cong F(c)$, which is natural in both $c$ and $F$.
\end{lemma}

\begin{lemma}[Right unit]
	For every functor $W : \cC \rightarrow \Set$, there exists an isomorphism $W(c)  \cong W \star \Yo(c) $, which is natural in both $c$ and $W$.
\end{lemma}

We now briefly discuss the dual theory of weighted limits. The dual definition and properties of weighted limits are analogous to that of weighted colimits.

\begin{definition}[Weighted limit]
	Let $W : \cC \rightarrow \Set$ and $F :\cC \rightarrow \cD$ be two arbitrary functors. The weighted limit $\{W,F\}$, if it exists, is defined by the following natural isomorphism in $d \in \cD$:
	\[\cD(d,\{W,F\}) \cong \Set^{\cC}(W,\cD(d,F(-)))\cdpunct{.}\]
\end{definition}

\begin{corollary}
	When all the weighted limits exist, they assemble into a $2$-ary functor $\{-,-\} : (\Set^{\cC})^{\op} \times \cD^{\cC} \rightarrow \cD$, which is cocontinuous in the weights argument. That is, the canonical comparison map $\{\Colim\limits_{i} W_i, F\} \rightarrow \Lim\limits_{i} \{W_i,F\}$ is an isomorphism.
\end{corollary}

However, weighted limits can be dualized in another evident way by declaring that the weighted limit of a functor $F$ is equivalent to defining the weighted colimit of the opposite functor $F^{\op}$. This is directly checked by manipulating the defining natural isomorphism of weighted limits and colimits.

\begin{lemma}
	A weighted limit of a functor $F : \cC \rightarrow \cD$ with respect to a weight $W : \cC \rightarrow \Set$ is equivalently defined as the weighted colimit of $F^{\op} : \cC^{\op} \rightarrow \cD^{\op}$ with respect to the same weight $W : \cC \rightarrow \Set$. Let $- \star - : \Set^{\cC} \times (\cD^{\cC})^{\op} \rightarrow \cD^{\op}$ and $\{-,-\}:(\Set^{\cC})^{\op} \times \cD^\cC \rightarrow \cD$ be the two defining functors. Then, there exists a natural isomorphism:
	\[\left(- \star -\right)^{\op} \cong \{-,-\}: (\Set^{\cC})^{\op} \times \cD^\cC \rightarrow \cD \cdpunct{.} \]
\end{lemma}

One of the most important example of weighted limits include the set of natural transformation between two functors $F,G : \cC \rightarrow \Set$. In particular, there is a natural isomorphism $\Set(X,\Set^{\cC}(F,G)) \cong \Set^{\cC}(F,\Set(X,G-))$ (natural in all $F,G,X$) that exhibits $\Set^{\cC}(F,G)$ as the weighted limit of $G$ with respect to $F$. This follows from cartesian closure of $\Set$ applied pointwise at $c \in \cC$.

The only substantial difference of weighted limits from weighted colimits is in the unitality property of $\Yo$. In this case, $\Yo$ behaves as a unit only in the following sense.

\begin{lemma}[left unit]
	For every $F: \cC \rightarrow \cD$, there exists an isomorphism $\{\cC(c,-),F\} \cong F(c)$, which is natural in both $F$ and $c \in \cC$. 
\end{lemma}
\begin{proof}
	This is a restatement of the Yoneda lemma with the observation that the set of natural transformation can be expressed as a weighted limit.
\end{proof}

Finally, the following examples illustrates how weighted limits and colimits generalize the ordinary ones. In essence, the weights dictate the \enquote{shape} of the cone point. Thus, we retrieve the usual definition of limits or colimits when we take the \enquote{shape} to be an actual single point.

\begin{example}[Ordinary limits]
	Let $\Delta *: \cC \rightarrow \Set$ be the constant functor at the terminal object $* \in \Set$. For every functor $F : \cC \rightarrow \cD$, there exists a natural isomorphism $\{\Delta *, F\} \cong \Lim F$, whenever the desired limits exist.
\end{example}
\begin{proof}
	By definition of ordinary limits using the category of cones, we have the natural isomorphism $\cD(d,\lim F) \cong \cD^{\cC}(\Delta d, F(-))$. Now, we observe that $\Set^{\cC}(\Delta *, \cD(d,F(-))) \cong \cD^{\cC}(\Delta d,F(-))$ by applying cartesian closure of $\Set$ pointwise. Hence, we are done.
\end{proof}

\begin{example}[Ordinary colimits]
	Let $\Delta *: \cC^{\op} \rightarrow \Set$ be the constant functor at the terminal object $* \in \Set$. For every functor $F : \cC \rightarrow \cD$, there exists a natural isomorphism ${\Delta*} \star F\cong \Colim F $, whenever the desired colimits exist.
\end{example}

\begin{remark}
	On the surface, it may seem that weighted limits strictly generalize ordinary limits. However, it turns out that this is not true in $\Set$. There is another characterization of weighted limits using the Grothendieck construction on weights, which shows that weighted limits can be expressed completely using ordinary limits. We provide a complete description in \cref{lemma:wlimislim}.
\end{remark}

\subsection{Kan Extensions}


Extension and lifting problems are ubiquitous in mathematics. In categorical context, for a given functor $F : \cC \rightarrow \cD$, we are primarily interested in the extension of the domain or the lift of the codomain. We only require the \emph{extension problems}. Let $G : \cC \rightarrow \cK$ be any other functor. Any diagram of the form
\[\begin{tikzcd}
	&& \cC \\
	\\
	\cK &&&& \cD
	\arrow["G"', from=1-3, to=3-1]
	\arrow[""{name=0, anchor=center, inner sep=0}, "F", from=1-3, to=3-5]
	\arrow[""{name=1, anchor=center, inner sep=0}, "{ \exists H}"', dashed, from=3-1, to=3-5]
	\arrow["{\exists \gamma}", shift right=5, shorten <=7pt, shorten >=7pt, Rightarrow, dashed, from=0, to=1]
\end{tikzcd} \cdpunct[6pt]{,}\]
exhibits $H : \cK \rightarrow \cD$ as a left extension of $F : \cC \rightarrow \cD$ using the $2$-cell $\gamma : F \Rightarrow HG$ and satisfies or solves the \emph{left extension problem} of $F$ along $G$. The dual \emph{right extension problems} are given by reversing the direction of $2$-cell or by applying the $(-)^{\op}$ involution. As with any categorical problem, we prefer to have a universal solution for the left extension problem. In this case, we want a universal pair $(H,\gamma)$ such that every other pair factors through it. This motivates the following definition of a left Kan extension.

\begin{definition}\label{defn:lkanext}
	The left Kan extension $\Lan_{G} F$ of $F$ along $G$ is given by a universal pair $(\Lan_{G} F, \eta)$. This universal pair satisfies the property that for any other pair $(H,\gamma)$ satisfying the left extension problem of $F$ along $G$, we have the diagram
	\[\begin{tikzcd}
		&& \cC &&&&&& \cC \\
		&&&&& {= } \\
		\cK &&&& \cD && \cK &&&& \cD
		\arrow["G"', from=1-3, to=3-1]
		\arrow[""{name=0, anchor=center, inner sep=0}, "F", from=1-3, to=3-5]
		\arrow["G"', from=1-9, to=3-7]
		\arrow[""{name=1, anchor=center, inner sep=0}, "F", from=1-9, to=3-11]
		\arrow[""{name=2, anchor=center, inner sep=0}, "H"', from=3-1, to=3-5]
		\arrow[""{name=3, anchor=center, inner sep=0}, "{\Lan_{G} H}", from=3-7, to=3-11]
		\arrow[""{name=4, anchor=center, inner sep=0}, "H"', curve={height=30pt}, from=3-7, to=3-11]
		\arrow["\gamma", shift right=5, shorten <=7pt, shorten >=7pt, Rightarrow, from=0, to=2]
		\arrow["\eta"', shift right=5, shorten <=7pt, shorten >=7pt, Rightarrow, from=1, to=3]
		\arrow["{\exists! \tilde \gamma}"', shorten <=4pt, shorten >=4pt, Rightarrow, dashed, from=3, to=4]
	\end{tikzcd} \cdpunct[2\baselineskip]{,} \]
exhibiting the factorization $\gamma = \tilde\gamma G \circ \eta$.
Whenever the universal pair $(\Lan_{G} F,\eta)$ exists as above, we say that it satisfies or solves the \emph{left Kan extension problem} of $F$ along $G$.
\end{definition} 

\begin{remark}
	One can also contrast this definition with the \emph{worst solution} to the left extension problem, provided that $\cD$ has a terminal object. Indeed, simply define the worst extension by sending all of $\cK$ to the terminal object of $\cD$, whenever it exists.
\end{remark}

As with any definition involving universal objects, we can reinterpret this definition of a Kan extension through natural isomorphisms.

\begin{lemma}
	\begin{enumerate}[leftmargin=*,labelsep=.5em,itemsep=0.5em,topsep=0em, parsep=0.25em]
		\item[]
		\item The left Kan extension of $F$ along $G$ is equivalently defined as a left extension pair $(\Lan_{G} F, \eta)$ that satisfies the following isomorphism for each $H$:
		\[\begin{tikzcd}
			& {\cD^{\cC}(\Lan_{G}{F} \circ G,H\circ G)} \\
			\\
			{\cD^{\cK}(\Lan_{G} F,H) } && {\cD^{\cC}(F,H \circ G)}
			\arrow["{\eta^*}", from=1-2, to=3-3]
			\arrow["{\cD^G}", from=3-1, to=1-2]
			\arrow["\cong"', from=3-1, to=3-3]
		\end{tikzcd} \cdpunct[6pt]{.}\] 
		\item The left Kan extension of $F$ along $G$ is equivalently defined by an object $\Lan_{G} F$ such that there exists an isomorphism $\cD^{\cK}(\Lan_{G} F, H) \cong \cD^{\cC}(F,H \circ G)$, which is natural in $H$.
	\end{enumerate}
	Moreover, $\Lan_{G} F$ is functorial in both $F$ and $G$, provided that the desirable left Kan extensions exist. 
\end{lemma}
\begin{proof}
	\begin{enumerate}[leftmargin=*,labelsep=.5em,itemsep=0.5em,topsep=0em, parsep=0.25em]
		\item []
		\item The existence of the unique $\tilde\gamma$ corresponding to each $\gamma$ provides the desired inverse map. Moreover, both the maps $\cD^{G}$ and $\eta_*$ are already natural in $H$. 
		\item The existence of any isomorphism $\cD^{\cK}(\Lan_{G} F, H) \cong \cD^{\cC}(F,H\circ G)$ natural in $H$, provides the desired $\eta$ by considering the image of $1_{\Lan_{G} F}$ under this isomorphism.
	\end{enumerate}
	The construction $\Lan_{G}F$ is covariantly functorial in $F$ and contravariantly functorial in $G$. This follows using the Yoneda lemma on the second characterization of the left Kan extension involving natural isomorphism in $H$.
\end{proof}

\begin{remark}
	Analogous definitions exist for right Kan extension by appropriately $(-)^{\op}$ dualizing.
\end{remark}

Unfortunately, these definitions of Kan extension are too abstract for most computational purposes. Thus, we provide other similar definitions which are more computable. These new definitions are mainly motivated by the idea that universal constructions such as limits/colimits can be defined representably in terms of limits/colimits in $\Set$. Hence, we ought to desire the same for Kan extensions.

\begin{definition}\label{defn:rkpoint}
	A right extension pair $(T, \eta)$ is called a pointwise\footnote{The usual definition of a pointwise Kan extension also assumes that it already is a Kan extension. However, this assumption is redundant and obscures the explicit analogy with the representable definition of ordinary limits and colimits.} right Kan extension if for all corepresentable functors given as $\cD(d,-) : \cD \rightarrow \Set$, the pair $(\cD(d,-) \circ T , \cD(d,-) \eta)$ is a right Kan extension pair.
\end{definition}

\begin{definition}\label{defn:lkpoint}
	A left extension pair $(T, \eta)$ is called a pointwise left Kan extension if for all representable functor given as $\cD(-,d) : \cD \rightarrow \Set^{\op}$, the pair $(\cD(-,d) \circ T , \cD(-,d) \eta)$ is a left Kan extension pair.
\end{definition}

The following two lemmas, adapted from the exposition in \cite{kellyect}, justify our claims about the computability and clarify the significance of the term \emph{pointwise}. Although these lemmas are formulated there for a general base of enrichment, the same holds in the case of $\Set$. Nevertheless, we provide explicit proofs for $\Set$ since the techniques used in the cited work lie far outside the scope of this paper. Furthermore, specializing their general base $\cV$ to $\Set$ is not always straightforward.

\begin{lemma}
	For every pointwise right Kan extension, if it exists, there is an isomorphism $T (k) \cong \{\cK(k,G-), F\}$, for each $k \in \cK$ naturally. In particular, the given weighted limits exist.
\end{lemma}

\begin{proof}
	We are given that the pair $(\cD(d,-) \circ T , \cD(d,-) \eta)$ is a right Kan extension for each $\cD(d,-)$. The following diagram depicts this pictorially:
	\[\begin{tikzcd}
		&& \cC \\
		\\
		\cK &&&& \cD && \Set
		\arrow["G"', from=1-3, to=3-1]
		\arrow[""{name=0, anchor=center, inner sep=0}, "F", from=1-3, to=3-5]
		\arrow[""{name=1, anchor=center, inner sep=0}, "T"', from=3-1, to=3-5]
		\arrow["{\cD(d,-)}", from=3-5, to=3-7]
		\arrow["\eta"', shift left=5, shorten <=7pt, shorten >=7pt, Rightarrow, from=1, to=0]
	\end{tikzcd}\cdpunct[8pt]{.}\]
	In particular, the pair satisfies the following universal property natural in $H : \cK \rightarrow \Set$:
	\[\begin{tikzcd}
		& {\Set^{\cC}(H\circ G,\cD(d,-)\circ T \circ G)} \\
		\\
		{\Set^{\cK}(H,\cD(d,-)\circ T) } && {\Set^{\cC}(H \circ G,\cD(d,-)\circ F)}
		\arrow["{(\cD(d,-)\eta)_*}", from=1-2, to=3-3]
		\arrow["{\Set^G}", from=3-1, to=1-2]
		\arrow["\cong"', from=3-1, to=3-3]
	\end{tikzcd} \cdpunct[8pt]{.}\]
	The above isomorphism is also natural in $d \in \cD$. Thereafter, by setting $H=\cK(k,-)$ and applying the Yoneda lemma, we obtain the following natural isomorphism in $d \in \cD$ and $k \in \cK$:
	\[\cD(d,T(k)) \cong \Set^{\cC}(\cK(k,G-),\cD(d,F-))\cdpunct{.}\]
	However, this is precisely the defining natural isomorphism of a weighted limit. In particular, we obtain the desired isomorphism $T(k)\cong \{\cK(k,G-), F\}$ natural in $k \in \cK$ by definition of a weighted limit.
\end{proof}

The more important part is that it is enough to allow all the desired weighted limits to exist.

\begin{lemma}[cf.~{\cite[Prop~4.46 and Thm.~4.6]{kellyect}}]
	If all the weighted limits $\{K(k,G-),F\}$ exist, they assemble into a functor $\Ran_{G} F(k) := \{K(k,G-),F\}$, which satisfies the universal property of being a right Kan extension of $F$ along $G$. Moreover, any right Kan extension defined this way must also be a pointwise right Kan extension.
\end{lemma}
\begin{proof}
	We show that there exists an isomorphism $\cD^{\cK}(H,\Ran_{G} F) \cong \cD^{\cC}(H\circ G,F)$ which is natural in $H$. Note that an element in $\cD^{\cK}(H,\Ran_{G} F)$ is simply a collection of maps $Hk \rightarrow \Ran_{G} F(k) \in \cD$ which is natural in $k \in \cK$. We can exhibit this fact using the concept of an \emph{end}. That is, we write $\cD^{\cK}(H,\Ran_{G} F) := \int_{k \in \cK}\cD(Hk, \Ran_{G} F(k))$\footnote{We haven't formally introduced ends, but for our purpose, they are nothing more than a book keeping tool for the naturality of a family of maps. For interested readers, see \cite{foscoend}.}. Now, we have the following chain of natural isomorphisms in $H$:
	\[\begin{tikzcd}
		{\cD^{\cK}(H,\Ran_{G} F)} & {\int_{k \in \cK}\cD(Hk,\Ran_{G} F(k))} \\
		& {\int_{k \in \cK} \Set^{\cC}(\cK(k,G-),\cD(Hk,F-))} \\
		& {\int_{k \in \cK} \int_{c \in \cC}\Set(\cK(k,Gc),\cD(Hk,Fc))} \\
		& {\int_{c \in \cC} \Set^{\cK}(\cK(-,Gc),\cD(H-,Fc))} \\
		{\cD^{\cC}(H \circ G,F)} & {\int_{c \in \cC}\cD(HGc,Fc)}
		\arrow[no head, from=1-1, to=1-2]
		\arrow[no head, from=1-2, to=2-2]
		\arrow[no head, from=2-2, to=3-2]
		\arrow[no head, from=3-2, to=4-2]
		\arrow[no head, from=4-2, to=5-2]
		\arrow[no head, from=5-2, to=5-1]
	\end{tikzcd} \cdpunct[8pt]{.}\]
	Hence, this demonstrates that $\Ran_{G} F(k) := \{\cK(k,G-),F\}$ satisfies the universal property of right Kan extension of $F$ along G. We are only left to show that this is a pointwise right Kan extension. 
	
	For each $k \in \cK$, let $\lambda^k$ denote the universal weighted cone corresponding to the weighted limit $\{\cK(k,G-),F\}$. Since weighted limits are preserved by representables (see \cite[Prop.~6.6.11]{borceux2}), it follows that the pair $(\cD(d, \{\cK(k,G-),F\}), \cD(d,-)_* \lambda^k)$ is the weighted limit of $\cD(d,F-)$ with weight $\cK(k,G-)$. Thus, if $(\{\cK(k,G-),F\} , \eta)$ is the original right Kan extension pair, this implies that the pair $(\cD(d, \{\cK(k,G-),F\}) , \cD(d,-)\eta)$ is also a right Kan extension pair as desired: To verify this, take $H= \Ran_{G} K$ and evaluate the defining chains of natural isomorphisms for both the original and the new weighted limits at the identity map $1_{\Ran_{G}K}$.
\end{proof}

\begin{corollary}
	A pointwise right Kan extension $(\Ran_{G} F, \eta)$ exists iff for all $k \in \cK $ the weighted limits $\{\cK(k,G-),F\}$ exist. In this case, $\Ran_{G} F$ is a right Kan extension and is computed at each $k \in \cK$ using the following weighted limit: 
	\[\Ran_{G} F (k) \cong \{\cK(k,G-), F\}\cdpunct{.}\]
\end{corollary}
\begin{proof}
	We have shown above that $\Ran_{G}F(k) \cong \{\cK(k,G-),F\}$ naturally in $k \in \cK$. However, the latter also satisfies the universal property of being a right Kan extension and we are done.
\end{proof}

\begin{remark}
	By the previous corollary, it is justified to call any pointwise right Kan extension $T$ simply as $\Ran_{G} F$, whenever it exists. However, the converse of this corollary isn't always true. That is, not every right Kan extension is a pointwise right Kan extension \cite[p.~65]{kellyect}. 
\end{remark}

In general, dual version of all the characterizations of pointwise right kan extension also hold for left Kan extension by $(-)^{\op}$ involution.

\begin{lemma}
	A pointwise left Kan extension $(\Lan_{G} F, \eta)$ exists iff for all $k \in \cK $ the weighted colimits $\cK(G-,k) \star F$ exist. In this case, $\Lan_{G} F$ is a left Kan extension and is computed at each $k \in \cK$ using the following weighted colimit:
	\[\Lan_{G} F (k) \cong \cK(G-,k) \star F \cdpunct{.}\]
\end{lemma}

Finally, we have the following lemma showing that weighted limits can be expressed completely using certain ordinary limits and further justify the claim that pointwise Kan extension are more computable.

\begin{lemma}[cf.~{\cite[Lemma.~6.6.1]{borceux2}}] \label{lemma:wlimislim}
	Let $\pi : \el W \rightarrow \cC $ denote the Grothendieck construction on the functor $W : \cC \rightarrow \Set$. For every $F : \cC \rightarrow \cD$, there exists an isomorphism $\{W,F\} \cong \Lim\limits_{\el W} F\pi$. Furthermore, this isomorphism is natural\footnote{The cited text doesn't claim or demonstrate the naturality. We include the verification here for completeness.} in both $F$ and $W$. 
\end{lemma}

\begin{proof}
	The essential result is that the left Kan extension of the constant functor at the terminal set $\Delta * :\el W \rightarrow \Set$ along the functor $\pi : \el W \rightarrow \cC$ is the functor $W : \cC \rightarrow \Set$ itself. That is, there is a left Kan extension pair $(W,\eta)$ as demonstrated in the following diagram:
	\[\begin{tikzcd}
		&& {\el W} \\
		\\
		\cC &&&& \Set
		\arrow["\pi"', from=1-3, to=3-1]
		\arrow[""{name=0, anchor=center, inner sep=0}, "{\Delta*}", from=1-3, to=3-5]
		\arrow[""{name=1, anchor=center, inner sep=0}, "W"', from=3-1, to=3-5]
		\arrow["\eta"', shift right=4, between={0.2}{0.8}, Rightarrow, from=0, to=1]
	\end{tikzcd}\cdpunct[8pt]{.}\]
	We define the component of $\eta$ at $(c,* \xrightarrow{x} W(c))$ as the morphism $* \xrightarrow{x} W(c)$. To verify the universal property, let $(H,\gamma)$ be any other left extension pair. Define $\tilde\gamma_{c} : W(c) \rightarrow H(c)$ by assigning each element $x \in W(c)$ to $\gamma_{c,x}(*) \in H(c)$. Thereafter, it is a routine check.
	
	Now, the universal property of $W$ as the left Kan extension provides the following natural isomorphism:
	\[\begin{tikzcd}
		&& {\Set^{\el W}(W \circ \pi, H \circ \pi)} \\
		\\
		{\Set^{\cC}(W, H)} &&&& {\Set^{\el W}(\Delta *, H \circ \pi)}
		\arrow["{\eta^*}", from=1-3, to=3-5]
		\arrow["{\Set^{\pi}}", from=3-1, to=1-3]
		\arrow["\cong"', from=3-1, to=3-5]
	\end{tikzcd}\cdpunct[8pt]{.}\]
	Set $H = \cD(d,F-)$. Then, a representative for the left side is the usual definition of $\{W,F\}$ and a representative for the right side is $\Lim F\pi$. 
	
	To prove the naturality, we first note that the naturality in $F$ follows directly from the naturality in $H$. Hence, we focus only upon the naturality in $W$. Let $\alpha : W_1 \Rightarrow W_2$ be arbitrary and recall from \cref{lemma:grfunct} that it induces a map $\el \alpha : \el W_1 \rightarrow \el W_2$ compatible with the projection $\pi : \el W \rightarrow \cC$. The maps $\alpha$ and $\el \alpha$ satisfy the relation $\eta_2 \el \alpha = \alpha\pi_1 \circ \eta_1$ as exhibited by the following diagram:
	\[\begin{tikzcd}
		&& {\el W_1} &&&&&& {\el W_1} \\
		\\
		&& {\el W_2} &&& {=} \\
		\\
		\cC &&&& \Set && \cC &&&& \Set
		\arrow["{\el \alpha}"{description}, from=1-3, to=3-3]
		\arrow["{\pi_1}"{description}, from=1-3, to=5-1]
		\arrow["{\Delta*}"{description}, from=1-3, to=5-5]
		\arrow["{\pi_1}"', from=1-9, to=5-7]
		\arrow[""{name=0, anchor=center, inner sep=0}, "{\Delta*}", from=1-9, to=5-11]
		\arrow["{\pi_2}"{description}, from=3-3, to=5-1]
		\arrow[""{name=1, anchor=center, inner sep=0}, "{\Delta*}"{description}, from=3-3, to=5-5]
		\arrow[""{name=2, anchor=center, inner sep=0}, "{W_2}"{description}, from=5-1, to=5-5]
		\arrow[""{name=3, anchor=center, inner sep=0}, "{W_1}", from=5-7, to=5-11]
		\arrow[""{name=4, anchor=center, inner sep=0}, "{W_2}"{description}, curve={height=30pt}, from=5-7, to=5-11]
		\arrow["{\eta_1}", shift right=4, between={0.2}{0.8}, Rightarrow, from=0, to=3]
		\arrow["{\eta_2}"', shift right=5, between={0.2}{0.8}, Rightarrow, from=1, to=2]
		\arrow["\alpha", between={0.2}{0.8}, Rightarrow, from=3, to=4]
	\end{tikzcd}\cdpunct[28pt]{.}\]
	Then, we use this relation to assert that the following diagram commutes:
	\[\begin{tikzcd}
		{\Set^{\cC}(W_2,H)} && {\Set^{\el W_2}(W_2\circ \pi_2, H \circ \pi_2)} && {\Set^{\el W_2}(\Delta *, H \circ \pi_2)} \\
		\\
		&& {\Set^{\el W_1}(W_2\circ \pi_1, H \circ \pi_1)} \\
		\\
		{\Set^{\cC}(W_1,H)} && {\Set^{\el W_1}(W_1\circ \pi_1, H \circ \pi_1)} && {\Set^{\el W_1}(\Delta *, H \circ \pi_1)}
		\arrow["{\Set^{\pi_2}}", from=1-1, to=1-3]
		\arrow["{\alpha^*}"', from=1-1, to=5-1]
		\arrow["{\eta_2^*}", from=1-3, to=1-5]
		\arrow["{\Set^{\el \alpha}}", from=1-3, to=3-3]
		\arrow["{\Set^{\el \alpha}}", from=1-5, to=5-5]
		\arrow["{(\alpha\pi_1)^*}", from=3-3, to=5-3]
		\arrow["{\Set^{\pi_1}}"', from=5-1, to=5-3]
		\arrow["{\eta_1^*}"', from=5-3, to=5-5]
	\end{tikzcd}\cdpunct[8pt]{.}\]
	This ensures that the natural isomorphism defining the left Kan extension pair $(W,\eta)$ is natural in the choice of $W$. Hence, the result follows.
\end{proof}

The dual version of the above lemma also holds.

\begin{corollary}
	Let $\pi : \el W \rightarrow \cC $ denote the Grothendieck construction on the functor $W : \cC^{\op} \rightarrow \Set$. For every $F : \cC \rightarrow \cD$, there exists an isomorphism $ W \star F \cong \Colim\limits_{\el W} F\pi$, which is natural $W, F$. 
\end{corollary}
\begin{proof}
	Use $\{W,F\}\cong W \star F^{\op}$ together with the duality of limits and colimits.
\end{proof}

The main consequence of these results is that ordinary completeness implies weighted completeness.
\begin{corollary}
	\begin{enumerate}[leftmargin=*,labelsep=.5em,itemsep=0.5em,topsep=0em, parsep=0.25em]
		\item[]
		\item If the category $\cD$ is complete then for all weights $X \in \Set^{\cC}$ and all functors $F \in \cD^{\cC}$, the weighted limits $\{W,F\}$ exist.
		\item If the category $\cD$ is cocomplete then for all weights $X \in \Set^{\cC^{\op}}$ and all functors $F \in \cD^{\cC}$, the weighted colimits $W\star F$ exist.
	\end{enumerate}
\end{corollary}

\subsection{Nerve and Realization}

Let $\cC$ be any small category, $\cD$ be any locally small category, and $F: \cC \rightarrow \cD$ be any functor between them. Using the Yoneda embedding $\Yo$ $: \cC \rightarrow \Set^{\cC^{\op}}$, we can consider two canonical left Kan extension problems. The first of them is depicted by the following diagram:
\[\begin{tikzcd}
	&& \cC \\
	\\
	{\Set^{\cC^{\op}}} &&&& \cD
	\arrow[""{name=0, anchor=center, inner sep=0}, "\Yo"', from=1-3, to=3-1]
	\arrow["F", from=1-3, to=3-5]
	\arrow[""{name=1, anchor=center, inner sep=0}, "N", dashed, from=3-5, to=3-1]
	\arrow[shift left=5, shorten <=8pt, Rightarrow, dashed, from=0, to=1]
\end{tikzcd}\cdpunct[8pt]{.}\]
Here, $N$ is the left Kan extension of $\Yo$ along the given functor $F: \cC \rightarrow \cD$. We call $N$ as the \emph{nerve functor}. This $N$ exists pointwise using the following formula for $N(d) \in \Set^{\cC^{\op}}$:
\[N(d)(-) \cong \cD(F,d) \star \Yo \cong \cD(F(-),d)\cdpunct{.}\]
$N(d)$ defines a functor $\cC^{\op} \rightarrow \Set$ using precomposition and $N$ itself defines the functor $N: \cD \rightarrow \Set^{\cC^{\op}}$ using postcomposition. Moreover, $N$ preserves all limits that exists in $\cD$. This is true because limits in $\Set^{\cC^{\op}}$ are computed pointwise.

\begin{lemma}
	The left Kan extension $N$ exists, is pointwise, preserves any limits that exist, and is given by the formula $N(d)(-) \cong \cD(F(-),d) : \cC^{\op} \rightarrow \Set$.
\end{lemma}

\begin{corollary}
	If the functor $F$ is fully faithful, then $N(F(c)) \cong \Yo(c)$.
\end{corollary}

The second left Kan extension problem is depicted by the following diagram:
\[\begin{tikzcd}
	&& \cC \\
	\\
	{\Set^{\cC^{\op}}} &&&& \cD
	\arrow["\Yo"', from=1-3, to=3-1]
	\arrow[""{name=0, anchor=center, inner sep=0}, "F", from=1-3, to=3-5]
	\arrow[""{name=1, anchor=center, inner sep=0}, "h"', dashed, from=3-1, to=3-5]
	\arrow[shift right=5, shorten <=7pt, Rightarrow, dashed, from=0, to=1]
\end{tikzcd}\cdpunct[8pt]{.}\]
In this problem, the existence of the left Kan extension $h$ is not guaranteed as $\cD$ is not cocomplete. However, when the necessary colimits exist, it is defined pointwise as follows:
\[h(X) \cong \Set^{\cC^{\op}}(\Yo,X) \star F \cong  X \star F \cdpunct{.}\]
We call this $h$ as the \emph{realization functor}. More importantly, when $h$ exists as above, it is necessarily the left adjoint to the nerve functor $N$. This follows from the following defining natural isomorphisms of weighted colimits:
\[\Set^{\cC^{\op}}(X,N(d)) \cong \Set^{\cC^{\op}}(X,\cD(F,d)) \cong \cD(X \star F,d) \cong \cD(h(X),d)\cdpunct{.}\]
Conversely, these natural isomorphisms implies that whenever $N$ has a left adjoint $h$, it must satisfy the universal property of being a weighted colimit. In particular, $h$ is defined by the same weighted colimits and satisfy being a pointwise left Kan extension. The following lemma encapsulate the key idea. 

\begin{lemma}
	A left adjoint $h \dashv N$ exists iff $h$ exists as a pointwise left Kan extension of $F$ along $\Yo$. In this situation, $h$ is necessarily given by the formula $h(X) \cong X \star F$. 
\end{lemma}

Finally, we note that the Yoneda embedding $\Yo$ is fully faithful. As a result, precomposing the left Kan extension with $\Yo$ should yield a natural isomorphism. Indeed, the isomorphism $h(\Yo(c)) \cong F(c)$ holds naturally in $c \in \cC$, whenever $h$ exists. However, it holds even without assuming the existence of $h$, as we have already shown that $\Yo(c) \star F \cong F(c)$ is true for any $F$.

	\section{Categories and simplicial sets}
	\subsection{\tpdf{Homotopy category $h$ and Nerve $N$}}\label{sec:handN}

Let $\bDelta$ be the full subcategory of $\Cat$ spanned by all $[n]$ with $n\geq 0$. We label objects of $\bDelta$ using $n$ instead of $[n]$. Now, we specialize the nerve and realization theory in the context of the inclusion $\bDelta \hookrightarrow \Cat$ to obtain the following two left Kan extension problems:
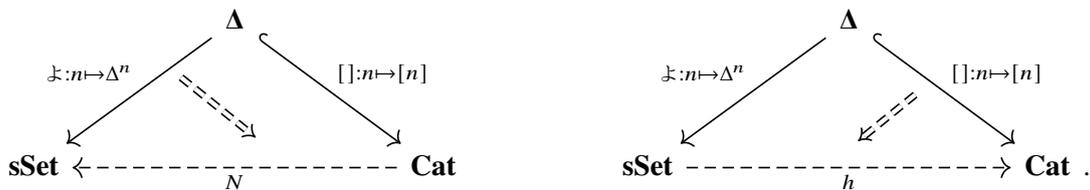
\[\begin{tikzcd}
	&& \bDelta &&&&&& \bDelta \\
	\\
	\sSet &&&& \Cat && \sSet &&&& \Cat
	\arrow[""{name=0, anchor=center, inner sep=0}, "{\Yo : n \mapsto \Delta^n}"', from=1-3, to=3-1]
	\arrow["{[] : n \mapsto[ n]}", hook, from=1-3, to=3-5]
	\arrow["{\Yo : n \mapsto \Delta^n}"', from=1-9, to=3-7]
	\arrow[""{name=1, anchor=center, inner sep=0}, "{[] : n \mapsto [n]}", hook, from=1-9, to=3-11]
	\arrow[""{name=2, anchor=center, inner sep=0}, "N", dashed, from=3-5, to=3-1]
	\arrow[""{name=3, anchor=center, inner sep=0}, "h"', dashed, from=3-7, to=3-11]
	\arrow[shift left=5, shorten <=8pt, Rightarrow, dashed, from=0, to=2]
	\arrow[shift right=2, shorten <=8pt, shorten >=8pt, Rightarrow, dashed, from=1, to=3]
\end{tikzcd}\cdpunct[8pt]{.}\]

$N$ exists as any general nerve does and is computed pointwise as $N(\cC)_{-} \cong \Cat([-],\cC)$. The face maps and degeneracy maps given by precomposition, and functoriality in $\cC \in \Cat$ given by postcomposition. 

For the realization $h$, we have no guarantee of its existence, unless we assume that the category $\Cat$ is cocomplete. Instead, we have the following corollaries which follow from specializing the general nerve and realization results.

\begin{corollary}
	A left adjoint $h \dashv N$ exists iff $h$ exists as a pointwise left Kan extension of $\bDelta \hookrightarrow \Cat$ along $\Yo$. In this situation, $h$ is given by the formula $h(X) \cong X \star []$.
\end{corollary}

\begin{corollary} \label{cor:Deltannat}
	For all $n \in \bDelta$, the weighted colimits $\Delta^n \star []$ exist and satisfy the isomorphism $\Delta^n \star [] \cong [n]$, natural in $n \in \bDelta$. 
\end{corollary}

The previous corollary together with the definition of weighted colimits immediately yield the isomorphism $\Cat([n],\cC) \cong \sSet(\Delta^n, N\cC)$, which is natural in both $n \in \bDelta$ and $\cC \in \Cat$. In essence, this suggests that $N$ is fully faithful. To begin with, we have the following important definitions.

\begin{definition}
	Define the $n$th spine $I_n$ to be the subsimplicial set of $\Delta^n$ spanned by all the non-degenerate $1$-simplices of the form $\{k,k+1\}$ for all $0\leq k \leq n-1$. 
\end{definition}

This simplicial set $I_n$ is also expressible as a gluing of $n$ copies of $\Delta^1$ along the edge points. That is, the following diagram is a colimiting cocone expressing $I_n$ as a colimit:
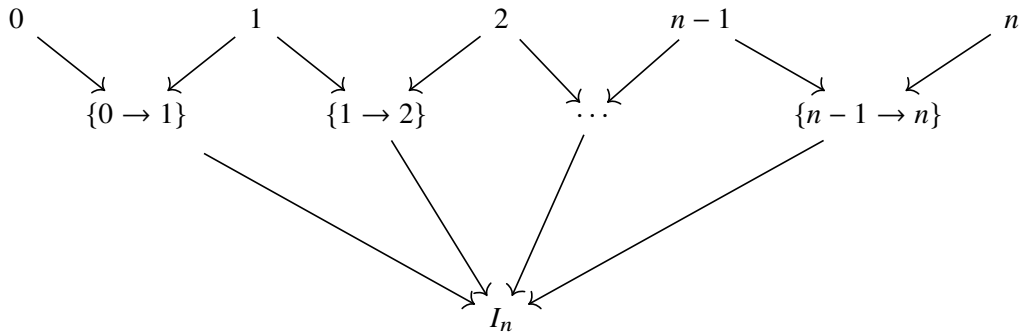
\[\begin{tikzcd}[column sep=small]
	0 && 1 && 2 && {n-1} && n \\
	& {\{0 \rightarrow 1\}} && {\{1 \rightarrow 2\}} && \cdots && {\{n-1 \rightarrow n\}} \\
	\\
	\\
	&&&& {I_n}
	\arrow[from=1-1, to=2-2]
	\arrow[from=1-3, to=2-2]
	\arrow[from=1-3, to=2-4]
	\arrow[from=1-5, to=2-4]
	\arrow[from=1-5, to=2-6]
	\arrow[from=1-7, to=2-6]
	\arrow[from=1-7, to=2-8]
	\arrow[from=1-9, to=2-8]
	\arrow[shorten <=10pt, from=2-2, to=5-5]
	\arrow[from=2-4, to=5-5]
	\arrow[from=2-6, to=5-5]
	\arrow[from=2-8, to=5-5]
\end{tikzcd}\cdpunct[8pt]{.}\]
Moreover, this diagram excluding the cone point also has a canonical map into $\Delta^n$. By the universal property of colimits, this provide a canonical inclusion map $i_n : I_n \hookrightarrow \Delta^n$.

\begin{definition} \label{defn:spextp}
	A simplicial set $X \in \sSet$ is said to have the $n$th spine extension property ($\iep_n$) if the canonical restriction map $\sSet(\Delta^n,X) \rightarrow \sSet(I_n,X)$ induced by $i_n: I_n \hookrightarrow \Delta^n$ is an isomorphism. In other words, the following extension diagram holds for each map $I_n \rightarrow X$:
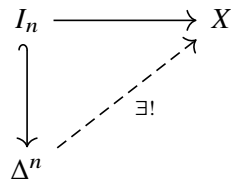
	\[\begin{tikzcd}
		{I_n} && X \\
		\\
		{\Delta^n}
		\arrow[from=1-1, to=1-3]
		\arrow[hook', from=1-1, to=3-1]
		\arrow["{\exists!}"', dashed, from=3-1, to=1-3]
	\end{tikzcd}\cdpunct[8pt]{.}\]
	If this property holds for all $n \geq 2$, we call it the \emph{spine extension property} $\iep$.
\end{definition}

\begin{remark}
	The property $\iep$ is trivially true for $n=0$ and $n=1$ because $I_{n}=\Delta^{n}$ in those cases. 
\end{remark}

The $\iep$ property plays an important role by characterizing the essential image of the nerve functor.

\begin{lemma} \label{lemma:spinemaps}
	Consider two maps $f,g : X \rightarrow Y$ in $\sSet$. If $f_1$ is a bijection then the map $f$ induces bijections $f_*|_{I_n}: \sSet(I_n,X) \xrightarrow{\cong} \sSet(I_n,Y)$ for all $n\geq 0$. If, moreover, $Y$ satisfies $\iep$ then the map $f$ is uniquely determined by its action on $1$-simplices, i.e., $f=g \iff f_1=g_1$.
\end{lemma}

\begin{proof}
	The diagram
	\[\begin{tikzcd}
		{\sSet(\Delta^n,X)} && {\sSet(\Delta^n,Y)} \\
		\\
		{\sSet(I_n,X)} && {\sSet(I_n,Y)}
		\arrow["{f_n= f_{*}|_{\Delta^n}}", from=1-1, to=1-3]
		\arrow["{i_n^*}"', from=1-1, to=3-1]
		\arrow["{i_n^*}", from=1-3, to=3-3]
		\arrow["{f_{*}|_{I_n}}", from=3-1, to=3-3]
	\end{tikzcd}\]
	commutes for all $n\geq 0$ as the order of postcomposition by $f$ or precomposition by $i_n$ does not matter. Because $I_n$ is a colimit of $n$ copies of $\Delta^1$, the map $f_*|_{I_n} : \sSet(I_n,X) \rightarrow \sSet(I_n,Y)$ for all $n\geq 0$ is uniquely determined by $f_1$. In other words, $f_*|_{I_n} = g_*|_{I_n}$ for all $n\geq 0$ iff $f_1 = g_1$. If furthermore, $f_1$ is a bijection then the map $f_*|_{I_n}$ is an isomorphism for all $n\geq 0$ by virtue of universal property of limits
	
	Now, the additional hypothesis that $Y$ satisfies $\iep$ provides the relation $f_*|_{\Delta^n} ={i_{n}^{*}}^{-1} \circ f_*|_{I_n} \circ i_{n}^{*}$ for all $n\geq 0$. In particular, we obtain the desired result that $f=g$ iff $f_1 = g_1$. 
\end{proof}

\begin{corollary}
	If both $X$ and $Y$ satisfy $\iep$ then the map $f$ is an isomorphism iff $f_1$ is a bijection.
\end{corollary}

\begin{proposition}\label{lemma:spine-N}
	For every $\cC \in \Cat$, $N\cC$ satisfies the spine extension property. Moreover, if $X \in \sSet$ satisfies the spine extension property then $X \cong N\cC$ for some $\cC \in \Cat$.
\end{proposition}

\begin{proof}
	We prove the first part by showing that $N\cC$ has the spine extension property. For $n=0$ and $n=1$, this is trivially true because $I_n \cong \Delta^n$. Therefore, we only consider $n\geq 2$.
	
	We show the existence of the unique spine extension by showing the weighted colimit $I_n \star []$ exists and that the canonical map $I_n \star [] \rightarrow \Delta^n \star []$ is an isomorphism. Because weighted colimits are cocontinuous in weights, we equivalently need to show that the following diagram is a colimiting cocone in $\Cat$:
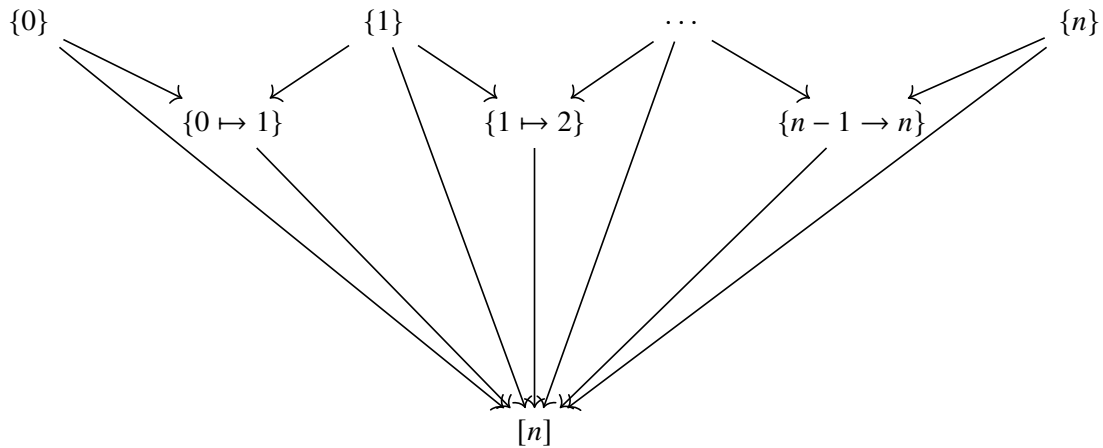
	\[\begin{tikzcd}[column sep=scriptsize]
		{\{0\}} &&& {\{1\}} && \cdots &&& {\{n\}} \\
		&& {\{0 \mapsto 1\}} && {\{1 \mapsto 2\}} && {\{n-1 \rightarrow n\}} \\
		\\
		\\
		\\
		\\
		&&&& {[n]}
		\arrow[from=1-1, to=2-3]
		\arrow[from=1-1, to=7-5]
		\arrow[from=1-4, to=2-3]
		\arrow[from=1-4, to=2-5]
		\arrow[from=1-4, to=7-5]
		\arrow[from=1-6, to=2-5]
		\arrow[from=1-6, to=2-7]
		\arrow[from=1-6, to=7-5]
		\arrow[from=1-9, to=2-7]
		\arrow[from=1-9, to=7-5]
		\arrow[from=2-3, to=7-5]
		\arrow[from=2-5, to=7-5]
		\arrow[from=2-7, to=7-5]
	\end{tikzcd}\cdpunct[8pt]{.}\]
	Thus, it is sufficient to show that specifying a morphism $[n] \rightarrow \cC$  is equivalent to specifying its action on all the objects $\{k\}$ and all the morphisms $\{k \rightarrow k+1\}$. This is true by definition of $[n]$.
	
	In the other direction, assume that $X \in \sSet$ satisfy the spine extension property. Define a category $\cC$ such that $\obj(\cC)=X_0$ and $\mor(\cC)=X_1$ with source, target, and identities derived from simplicial identities; and composition defined by the $\iep_2$ property. The only non-trivial check is that the composition is associative but this follows directly by first using $\iep_3$ to fill the $3$-simplex and then using $\iep_2$ for the uniqueness of composites.
	
	The construction of $\cC$ guarantees that $X_0 \cong N\cC_0$ and $X_1 \cong N\cC_1$. However, we do not yet have a \emph{simplicial}\footnote{In standard sources such as \cite{luriehtt2009}, the existence of a simplicial map is claimed to be self-evident from the construction of $\cC$. This is not entirely immediate to us, whether one uses spine extensions, as we do, or one uses horn extensions, as in the source. Thus, we provide a complete verification here using spine extensions, which we found to be simpler.} map $f: X \rightarrow N\cC$ that produces these bijections. We define the map on arbitrary $n$-simplices using the following composite:
	\[\begin{tikzcd}
		{\sSet(\Delta^n,X)} &&& {\sSet(\Delta^n,N\cC)} \\
		\\
		{\sSet(I_n,X)} &&& {\sSet(I_n,N\cC)}
		\arrow["{f_n :={i_n^*}^{-1} \circ f_{*}|_{I_n}\circ i_n^*}", from=1-1, to=1-4]
		\arrow["{i_n^* (\cong)}"', from=1-1, to=3-1]
		\arrow["{f_{*}|_{I_n}}"', from=3-1, to=3-4]
		\arrow["{{i_n^*}^{-1} (\cong)}"', from=3-4, to=1-4]
	\end{tikzcd}\cdpunct[8pt]{.}\]
	Here, the $n=0$ and $n=1$ case is defined by virtue of the bijections $X_0 \cong N\cC_0$ and $X_1 \cong N\cC_1$. The simplicial identities between $f_0$ and $f_1$ holds by our choice of category structure on $\cC$. And the map $f_*|_{I_n}$ is defined using $f_1$ and universal property of limits.
	
	First, we consider the compatibility between $f_{1}$, $f_{2}$ and $d_{i},s_{i}$ (i.e., $n=2$ case): 
	\[\begin{tikzcd}
		& {\sSet(\Delta^1,N\cC)} && {\sSet(\Delta^{2},N\cC)} \\
		\\
		{\sSet(\Delta^1,X)} && {\sSet(\Delta^{2},X)} & {\sSet(I_{2},N\cC)} \\
		\\
		&& {\sSet(I_{2},X)}
		\arrow["{s_i}"{description}, shift left=2, from=1-2, to=1-4]
		\arrow["{d_i}"{description}, shift left=2, from=1-4, to=1-2]
		\arrow["{i_{2}^*(\cong)}"{description}, from=1-4, to=3-4]
		\arrow["{f_1}"{description}, from=3-1, to=1-2]
		\arrow["{s_i}"{description}, shift left=2, from=3-1, to=3-3]
		\arrow["{f_2}"{description}, from=3-3, to=1-4]
		\arrow["{d_i}"{description}, shift left=2, from=3-3, to=3-1]
		\arrow["{i_{2}^*(\cong)}"{description}, from=3-3, to=5-3]
		\arrow["{f_*|_{I_2}}"{description}, from=5-3, to=3-4]
	\end{tikzcd}\cdpunct[8pt]{.}\]
	
	Because $i_2^*$ is an isomorphism, we can instead check the compatibilities between $f_1 = f_*|_{I_1}$ and $f_*|_{I_2}$. The desired result now follows because identity morphisms in $\cC$ are precisely the degenerate $1$-simplices of $X$ and because composition in $\cC$ is defined precisely using $\iep_2$ of $X$.
	
	
	We now focus on the $n\geq 3$ case. Define a simplicial set $I_{n,i}$ for all $0\leq i \leq n$ by setting $I_{n,0}=I_{n,n}=I_n$ and setting $I_{n,i}$ for $0 < i < n$ to be the following simplicial subset of $\Delta^n$:
	\[\begin{tikzcd}
		&&&& i \\
		0 & 1 & \cdots & {i-1} && {i+1} & \cdots & n
		\arrow[from=1-5, to=2-6]
		\arrow[from=2-1, to=2-2]
		\arrow[from=2-2, to=2-3]
		\arrow[from=2-3, to=2-4]
		\arrow[from=2-4, to=1-5]
		\arrow[from=2-4, to=2-6]
		\arrow[from=2-6, to=2-7]
		\arrow[from=2-7, to=2-8]
	\end{tikzcd}\cdpunct[6pt]{.}\]
	Here, the sphere formed by $\{i-1,i,i+1\}$ is filled with $\Delta^2$. In particular, for each $0 < i < n$, the simplicial set $I_{n,i}$ provides a canonical map $d^i : I_{n-1} \hookrightarrow I_{n,i}$ and the following auxiliary diagrams:
	\[\begin{tikzcd}
		{\Delta^{n-1}} && {\Delta^n} &&&& {\sSet(\Delta^n,X)} && \\
		\\
		{I_{n-1}} && {I_{n,i}} && {I_n} && {\sSet(I_{n,i},X)} && {\sSet(I_{n},X)} \\
		\\
		&& {\Delta^2} && {I_2} && {\sSet(\Delta^2,X)} && {\sSet(I_{2},X)}
		\arrow["{d^i}", hook, from=1-1, to=1-3]
		\arrow["{(2-\mathrm{o}-3)\cong}"', from=1-7, to=3-7]
		\arrow["{i_n^*(\cong)}", from=1-7, to=3-9]
		\arrow["{i_{n-1}}", hook', from=3-1, to=1-1]
		\arrow["{d^i}", hook, from=3-1, to=3-3]
		\arrow[hook, from=3-3, to=1-3]
		\arrow["\ulcorner"{anchor=center, pos=0.125}, draw=none, from=3-3, to=5-5]
		\arrow["{i_n}"', hook, from=3-5, to=1-3]
		\arrow[hook, from=3-5, to=3-3]
		\arrow["{(\mathrm{pb})\cong}", from=3-7, to=3-9]
		\arrow[from=3-7, to=5-7]
		\arrow["\ulcorner"{anchor=center, pos=0.125}, draw=none, from=3-7, to=5-9]
		\arrow[from=3-9, to=5-9]
		\arrow[hook', from=5-3, to=3-3]
		\arrow[hook, from=5-5, to=3-5]
		\arrow["{i_2}", hook', from=5-5, to=5-3]
		\arrow["{i_2^*(\cong)}", from=5-7, to=5-9]
	\end{tikzcd}\cdpunct[8pt]{.}\]
	
	
	Analogous to $d^i : I_{n-1} \hookrightarrow I_{n,i}$ as above, let $s^i : I_{n} \hookrightarrow I_{n-1}$ be the restriction of the map $s^i : \Delta^{n} \hookrightarrow \Delta^{n-1}$ along $i_n$ and $i_{n-1}$. We then have the following diagrams expressing the compatibility relations:
	
	\[\begin{tikzcd}
		& {\sSet(I_{n-1},N\cC)} && {\sSet(I_n,N\cC)} \\
		\\
		{\sSet(I_{n-1},X)} && {\sSet(I_{n},X)} \\
		\\
		& {\sSet(I_{n-1},N\cC)} & {\sSet(I_{n,i},N\cC)} & {\sSet(I_n,N\cC)} \\
		\\
		{\sSet(I_{n-1},X)} & {\sSet(I_{n,i},X)} & {\sSet(I_{n},X)}
		\arrow["{s_i}", from=1-2, to=1-4]
		\arrow["{f_*|_{I_{n-1}}}", from=3-1, to=1-2]
		\arrow["{s_i}"', from=3-1, to=3-3]
		\arrow["{f_*|_{I_n}}"', from=3-3, to=1-4]
		\arrow["{d_i}"', from=5-3, to=5-2]
		\arrow["\cong", from=5-3, to=5-4]
		\arrow["{f_*|_{I_{n-1}}}", from=7-1, to=5-2]
		\arrow["{(\mathrm{pb})\exists!}", dashed, from=7-2, to=5-3]
		\arrow["{d_i}"', from=7-2, to=7-1]
		\arrow["\cong", from=7-2, to=7-3]
		\arrow["{f_*|_{I_{n}}}"', from=7-3, to=5-4]
	\end{tikzcd}\cdpunct[8pt]{.}\]

	The desired compatibility relations then follows exactly for the same reasons as the $n=2$ case. Finally, we can transfer these compatibility relations to $f_{n-1} ,f_{n}$ using the spine extension isomorphisms because everything commutes. This finishes the proof.
\end{proof}

This lemma yields two important immediate corollaries.

\begin{corollary}
	The nerve functor $N$ is fully faithful.
\end{corollary}
\begin{proof}
	Every simplicial map $f : N\cC \rightarrow N\cD$ yields a functor $F : \cC \rightarrow \cD$ using $f_0, f_1$, and the simplicial identities. Since both $N\cC$ and $N\cD$ satisfy the $\iep$ property this ensures that $NF = f$ uniquely.
\end{proof}

\begin{corollary}
	The weighted colimit $N\cC \star []$ exists for any $\cC \in \Cat$ and satisfy $N\cC \star [] \cong \cC$ naturally in $\cC$. In particular, if some $X \in \sSet$ satisfies $X \cong N \cC$ then $X \star [] \cong \cC$.
\end{corollary}
\begin{proof}
	The defining natural isomorphism $\Cat(\cC,\cD) \cong \sSet(N\cC,\Cat([],\cD))$ of the fully faithful functor $N$
	implies that $\cC$ satisfies the universal property of the weighted colimit $N\cC \star []$.
\end{proof}

Because $N$ is fully faithful, the existence of the left adjoint $h$ also implies that $N$ is a \textit{reflective embedding}. As a result, we have an equivalence between the existence of the weighted colimits of the form $X \star []$ and the cocompleteness of $\Cat$ \cite[Prop.~4.5.15]{reihlcat}. Furthermore, the existence of the reflective embedding automatically ensures completeness. 

\begin{proposition}\label{prop:wlim}
	The category $\Cat$ is cocomplete iff for every $X \in \sSet$, the weighted colimit $X \star []$ exists. Furthermore, under these equivalent conditions, $\Cat$ is complete.
\end{proposition}

\subsection{\tpdf{Nerve and $2$-coskeleton}}

Our goal in this section is to discuss two canonical constructions to extract the $n$-dimensional information out of a simplicial set and describe how the nerve of a simplicial set is inherently $2$-dimensional. Afterwards, we explore the relevance of this $2$-dimensionality for the existence of $h$.

\begin{lemma}
	Consider the restriction functor $j_n : \bDelta_{\leq n} \hookrightarrow \bDelta$ of the simplex category given by the full subcategory on all ordinals $k \leq n$. There is the following triple of adjunction induced by the global left and right Kan extension:
	\[\begin{tikzcd}
		{\sSet^{\leq n}} &&& \sSet
		\arrow[""{name=0, anchor=center, inner sep=0}, "{\Sk_n}", curve={height=-18pt}, from=1-1, to=1-4]
		\arrow[""{name=1, anchor=center, inner sep=0}, "{\Cosk_n}"', curve={height=18pt}, from=1-1, to=1-4]
		\arrow[""{name=2, anchor=center, inner sep=0}, "{(-)_{\leq n}}"{description}, from=1-4, to=1-1]
		\arrow["\dashv"{anchor=center, rotate=-90}, draw=none, from=0, to=2]
		\arrow["\dashv"{anchor=center, rotate=-90}, draw=none, from=2, to=1]
	\end{tikzcd}\cdpunct{.}\]
\end{lemma}

\begin{proof}
	The category $\sSet$ is both complete and cocomplete. Therefore, the left and the right Kan extension exist for all $X_{\leq n} \in \sSet^{\leq n}$ along $(j_n)^{\op}$. In particular, $\Sk_n (X_{\leq n})(n') := \bDelta(n',j_n(-)) \star X(-)$ and $\Cosk_n (X_{\leq n})(n') := \{\bDelta(j_n(-),n'),X(-)\}$, where $n' \in \bDelta$.
\end{proof}

Moreover, the functor $j_n$ is fully faithful. In particular, this implies that the unit $\eta^{\Sk} : 1 \Rightarrow (-)_{\leq n} \circ \Sk_n$ and the counit $\epsilon^{\Cosk} : (-)_{\leq n} \circ \Cosk_n \Rightarrow 1$ are both natural isomorphisms. Equivalently, both $\Sk_n$ and $\Cosk_n$ are fully faithful functors. We also define two other composite functors $\sk_n = \Sk_n \circ (-)_{\leq n}$ and $\cosk_n =  \Cosk_n \circ (-)_{\leq n}$, which further fit into the adjunction 
\[\begin{tikzcd}
	\sSet && \sSet
	\arrow[""{name=0, anchor=center, inner sep=0}, "{\sk_n}", shift left=2, from=1-1, to=1-3]
	\arrow[""{name=1, anchor=center, inner sep=0}, "{\cosk_n}", shift left=2, from=1-3, to=1-1]
	\arrow["\dashv"{anchor=center, rotate=-90}, draw=none, from=0, to=1]
\end{tikzcd}\cdpunct[0pt]{,}\]
by virtue of composition of adjoints. These composed adjoints are called the \emph{$n$-skeleton} and \emph{$n$-coskeleton} functor respectively. Correspondingly, the functor $(-)_{\leq n}$ is called \emph{$n$-truncation}.

\begin{definition}\label{defn:skcosk}
	A simplicial set $X$ is called $n$-skeletal if the counit map $\Sk_n X_{\leq n}:= \sk_n X \rightarrow X$ of the adjunction $\Sk_n \dashv (-)_{\leq n}$ is an isomorphism. Similarly, $X$ is called $n$-coskeletal if the unit map $X \rightarrow \cosk_n X := \Cosk_n X_{\leq n}$ of the adjunction $(-)_{\leq n} \dashv \Cosk_n$ is an isomorphism.
\end{definition}

The embedding $N$ interacts nicely with the $2$-skeleton and $2$-coskeleton adjunction by virtue of $N$ being $2$-coskeletal\footnote{This is very standard to check by verifying the unique extension property along the maps $\partial \Delta^n \hookrightarrow \Delta^n$ for all $n\geq 3$.}, i.e., $\eta^{\Cosk_2}N$ is a natural isomorphism.  Then, we have the following lemma and corollary which capture the key result.

\begin{lemma}
	The map $\sk_2 X \star [] \rightarrow X \star []$, induced by the counit map $\epsilon^{\sk_2}_X : \sk_2 X \rightarrow X$ of the adjunction $\sk_2 \dashv (-)_{\leq 2}$, is an isomorphism. In particular, the weighted colimit $X \star []$ exists iff the weighted colimit $\sk_2 X \star []$ exists.
\end{lemma}
\begin{proof}
	We have the following diagram for all $X,Y \in \sSet$:
	\[\begin{tikzcd}
		{\sSet(X,Y)} &&&& {\sSet(X,\cosk_2 Y)} \\
		\\
		&&&& {\sSet(\sk_2 X, \sk_2\cosk_2 Y)} \\
		\\
		&&& {\sSet(\sk_2 X,\sk_2Y)} \\
		\\
		&&& {\sSet(\sk_2X,Y)}
		\arrow["{\eta^{\Cosk_2}_{Y}}", from=1-1, to=1-5]
		\arrow["{\sk_2}", from=1-1, to=5-4]
		\arrow["{\epsilon^{\sk_2}_X}"', from=1-1, to=7-4]
		\arrow["{\sk_2}", from=1-5, to=3-5]
		\arrow["{\Sk_2 \cdot\epsilon^{\Cosk_2}\cdot (-)_{\leq 2}}"', from=3-5, to=5-4]
		\arrow["{\epsilon^{\sk_2 \dashv \cosk_2}_Y}", from=3-5, to=7-4]
		\arrow["{\epsilon^{\Sk_2}_{Y}}"{description}, from=5-4, to=7-4]
	\end{tikzcd}\cdpunct[8pt]{.}\]

	Here, for the brevity of notation, we are \emph{suppressing} the postcompose and precompose scripts. 
	
	To show that the given diagram commutes, we note the following: The upper square commutes by applying $\Sk_2$ functor to the triangle inequality $\epsilon^{\cosk_2}(-)_{\leq 2} \circ (-)_{\leq 2} \eta^{\cosk_2} = 1_{(-)_{\leq 2}}$. The bottom right triangle is the definition of induced counit for the composite adjoint. Lastly, the bottom left triangle is the naturality of counit map $\epsilon^{\Sk_2}$.
	
	The map $\epsilon^{\sk_2 \dashv \cosk_2}_Y \circ \sk_2$ is simply the adjunction (hom set) isomorphism. If we take $Y$ to be $2$-coskeletal then $\eta^{\Cosk_2}_Y$ is also an isomorphism. Hence, the resultant map $\epsilon^{\sk_2}_X$ is an isomorphism. That is, $\sSet(X,Y)\xrightarrow{\epsilon^{\sk_2}_X} \sSet(\sk_2 X,Y)$ is an isomorphism as desired.
	
	Finally, by definition of weighted colimits, the induced map $\sk_2 X \star [] \rightarrow X \star []$ is an isomorphism and one of the weighted colimits exist iff the other one does.  
\end{proof}

\begin{corollary}\label{cor:sk2catcom}
	The category $\Cat$ is cocomplete iff for every $X \in \sSet$ the weighted colimit $\sk_2X \star []$ exists. Furthermore, $\Cat$ is complete under these equivalent conditions.
\end{corollary}

	\section{Existence of the left adjoint}
	\subsection{Weighted colimits and the left adjoint}\label{sec:ladj}

The \cref{cor:sk2catcom} ensures the existence of the left adjoint $h$, whenever the colimits $\sk_2 X \star []$ exist for all $X \in \sSet$. As a consequence, it suffices to work with $2$-skeletal filtration of simplicial sets. Recall that any simplicial set $X$ admits the following $2$-skeletal filtration by non-degenerate simplices \cite[p.~8--9]{jardinesht}: 
\[\begin{tikzcd}[column sep=scriptsize]
	{\Coprod\limits_{X_1^{\nd}}\partial\Delta^1} && {\Coprod\limits_{X^{\nd}_1} \Delta^1} && {\Coprod\limits_{X^{\nd}_2} \partial \Delta^2} && {\Coprod\limits_{X^{\nd}_2} \Delta^2} \\
	\\
	& {\sk_0X} && {\sk_1X} && {\sk_2X} && X
	\arrow[hook, from=1-1, to=1-3]
	\arrow[from=1-1, to=3-2]
	\arrow[from=1-3, to=3-4]
	\arrow[hook, from=1-5, to=1-7]
	\arrow[from=1-5, to=3-4]
	\arrow[from=1-7, to=3-6]
	\arrow[hook, from=3-2, to=3-4]
	\arrow["\ulcorner"{anchor=center, pos=0.125, rotate=180}, shift left=3, draw=none, from=3-4, to=1-3]
	\arrow[hook, from=3-4, to=3-6]
	\arrow["\ulcorner"{anchor=center, pos=0.125, rotate=180}, draw=none, from=3-6, to=1-5]
	\arrow[hook, from=3-6, to=3-8]
\end{tikzcd}\cdpunct[8pt]{.}\]
Our goal is to utilize this filtration to aid the computation of the weighted colimits $X \star []$. Recall that weighted colimits are cocontinuous in the weights argument and notice that the skeletal filtration diagram of $X$ is comprised entirely of colimits and composites of maps. Hence, we obtain the following diagram containing only weighted colimits, coproducts, and pushouts:
\[\begin{tikzcd}[column sep=tiny]
	{\Coprod\limits_{X^{\nd}_1} \partial \Delta^1 \star []} && {\Coprod\limits_{X^{\nd}_1} \Delta^1 \star []} && {\Coprod\limits_{X^{\nd}_2} \partial \Delta^2 \star []} && {\Coprod\limits_{X^{\nd}_2} \Delta^2 \star []} \\
	\\
	\\
	& {\sk_0 X \star []} && {\sk_1 X \star []} && {\sk_2 X \star  []} & {X \star []}
	\arrow[""{name=0, anchor=center, inner sep=0}, from=1-1, to=1-3]
	\arrow[from=1-1, to=4-2]
	\arrow[from=1-3, to=4-4]
	\arrow[from=1-5, to=1-7]
	\arrow[from=1-5, to=4-4]
	\arrow[from=1-7, to=4-6]
	\arrow[from=4-2, to=4-4]
	\arrow[from=4-4, to=4-6]
	\arrow["\ulcorner"{anchor=center, pos=0.125, rotate=180}, draw=none, from=4-6, to=1-5]
	\arrow["\cong", no head, from=4-6, to=4-7]
	\arrow["\ulcorner"{anchor=center, pos=0.125, rotate=180}, draw=none, from=4-4, to=0]
\end{tikzcd}\cdpunct[8pt]{.}\]

\newsavebox\twodots
\sbox\twodots
{
	\begin{tikzcd}[ampersand replacement=\&,column sep=scriptsize, baseline=(current bounding box.center)]
		\bullet \& \bullet
	\end{tikzcd}
}

\newsavebox\anarrow
\sbox\anarrow
{
	\begin{tikzcd}[ampersand replacement=\&,column sep=scriptsize, baseline=(current bounding box.center)]
		\bullet \& \bullet
		\arrow[from=1-1, to=1-2]
	\end{tikzcd}
}

\newsavebox\freetwosimp
\sbox\freetwosimp
{
	\begin{tikzcd}[ampersand replacement=\&,column sep=scriptsize, baseline=(current bounding box.center)]
		\& \bullet \\
		\bullet \&\& \bullet
		\arrow["{g \circ f}", curve={height=-6pt}, from=2-1, to=2-3]
		\arrow["f", from=2-1, to=1-2]
		\arrow["g", from=1-2, to=2-3]
		\arrow["h"', curve={height=6pt}, from=2-1, to=2-3]
	\end{tikzcd}
}

\newsavebox\twosimp
\sbox\twosimp
{
	\begin{tikzcd}[ampersand replacement=\&,column sep=scriptsize, baseline=(current bounding box.center)]
		\& \bullet \\
		\bullet \&\& \bullet
		\arrow[""{name=0, anchor=center, inner sep=0}, "{g \circ f}", curve={height=-6pt}, from=2-1, to=2-3]
		\arrow["f", from=2-1, to=1-2]
		\arrow["g", from=1-2, to=2-3]
		\arrow[""{name=1, anchor=center, inner sep=0}, "h"', curve={height=6pt}, from=2-1, to=2-3]
		\arrow[shorten <=2pt, shorten >=2pt, Rightarrow, no head, from=0, to=1]
	\end{tikzcd}
}
Unfortunately, we haven't demonstrated the existence of all the desired colimits in the above diagram. Thus, to conclude the existence of $X \star []$, we must establish the existence of all the auxiliary colimits and the two pushouts in the above diagram.

We first focus on all the colimits other than the two pushout squares.

\begin{lemma}
	The following auxiliary colimits/weighted colimits exist:
	\begin{enumerate} [labelsep=.5em,itemsep=0.5em,topsep=0em, parsep=0.25em]
		\item The coproduct $\coprod \cC_{i}$ for an indexed family $\{\cC_i\}_{i \in I}$ with $I \in \Set$.
		\item The weighted colimit $\sk_0 X \star []$.
		\item The weighted colimits $\Delta^1 \star []$ and $\Delta^2 \star []$.
		\item The weighted colimit $\partial \Delta^2 \star []$.
	\end{enumerate}
\end{lemma}
\begin{proof}
	\begin{enumerate}[labelsep=.5em,itemsep=0.5em,topsep=0em, parsep=0.25em]
		\item []
		\item The coproduct $\coprod \cC_i$ is simply given as the disjoint union of the set of objects and of the set of morphisms. This is because $\coprod N\cC_{i}$ satisfies the spine extension property (see \cref{lemma:spine-N}). In particular, $N$ also preserves these coproducts.
		\item The colimit $\sk_0 X \star []$ is computed using the isomorphism $\Coprod\limits_{X_0} (\Delta^0 \star []) \xrightarrow{\cong} \sk_0 X \star []$. This is just the discrete category on the set $X_0$. Similarly, the colimit $\partial \Delta^1 \star []$ is a just a discrete category with $2$ objects.
		\item The colimits $\Delta^1 \star []$ and $\Delta^2 \star []$ are already shown to exist and be isomorphic to $[1]$ and $[2]$ respectively (refer \cref{cor:Deltannat}).
		\item The colimit $\partial \Delta^2 \star []$ is the only non trivial one. However, $\partial \Delta^2$ can be constructed by gluing three $\Delta^1$ along the endpoints. Thereafter, each of the three $\Delta^1 \star [] \cong [1]$ glues similarly along the objects. This yields the resultant colimit as the following diagram:
		
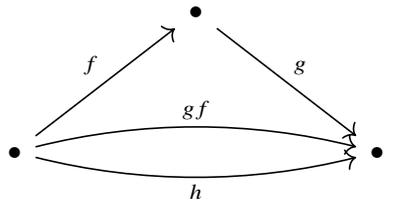
\[\begin{tikzcd}
			&& \bullet \\
			\\
			\bullet &&&& \bullet
			\arrow["g", from=1-3, to=3-5]
			\arrow["f", from=3-1, to=1-3]
			\arrow["gf", curve={height=-12pt}, from=3-1, to=3-5]
			\arrow["h"', curve={height=12pt}, from=3-1, to=3-5]
		\end{tikzcd}\cdpunct[16pt]{.}\]
		Here, the arrows $f,g,h$ are some independent morphisms. One verifies this by directly establishing the desired universal property of the said diagram.
	\end{enumerate}
\end{proof}

\begin{remark}
	The maps between these computed auxiliary colimits are also the desired ones. This follows from the naturality in $n \in \bDelta$ of the isomorphisms $\Delta^n \star [] \cong [n]$.
\end{remark}

In particular, the aforementioned auxiliary colimits exist and provide the following new diagram:
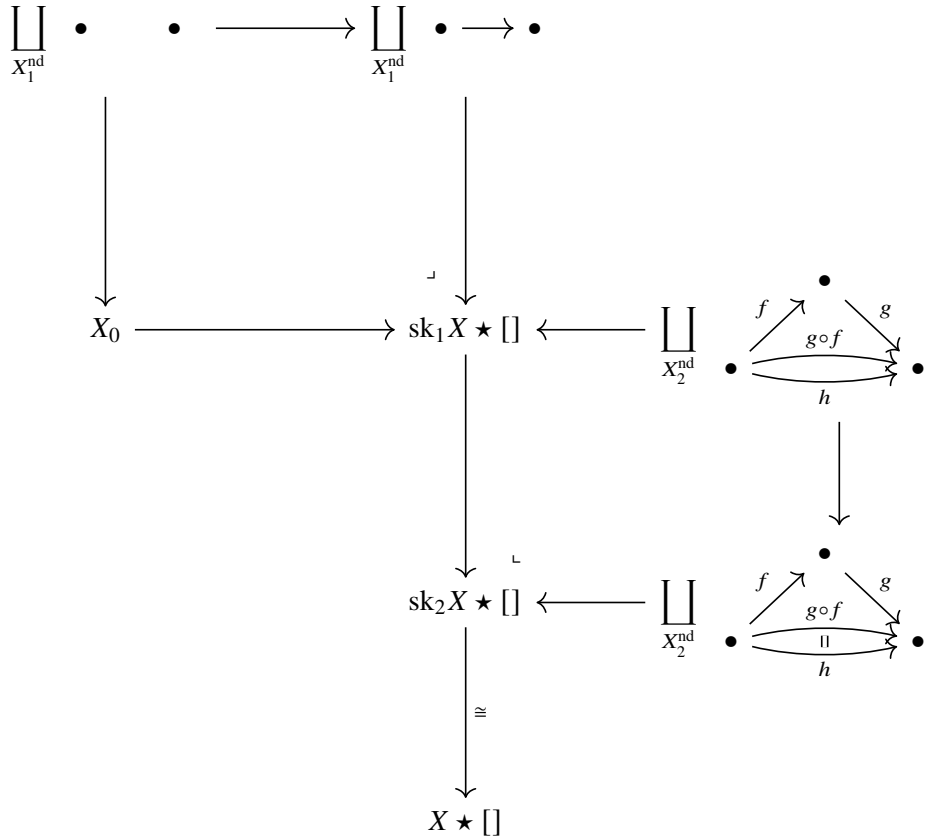
\[\begin{tikzcd}
	{\Coprod\limits_{X^{\nd}_1}\usebox{\twodots}} && {\Coprod\limits_{X^{\nd}_1} \usebox{\anarrow}} \\
	\\
	\\
	{X_0} && {\sk_1 X \star []} & {\Coprod\limits_{X^{\nd}_2} \usebox{\freetwosimp}} \\
	\\
	&& {\sk_2 X \star []} & {\Coprod\limits_{X^{\nd}_2} \usebox{\twosimp}} \\
	\\
	&& {X \star []}
	\arrow[""{name=0, anchor=center, inner sep=0}, from=1-1, to=1-3]
	\arrow[from=1-1, to=4-1]
	\arrow[from=1-3, to=4-3]
	\arrow[from=4-1, to=4-3]
	\arrow[from=4-3, to=6-3]
	\arrow[from=4-4, to=4-3]
	\arrow[shift left=5, from=4-4, to=6-4]
	\arrow["\ulcorner"{anchor=center, pos=0.125, rotate=90}, draw=none, from=6-3, to=4-4]
	\arrow["\cong", from=6-3, to=8-3]
	\arrow[from=6-4, to=6-3]
	\arrow["\ulcorner"{anchor=center, pos=0.125, rotate=180}, draw=none, from=4-3, to=0]
\end{tikzcd}\cdpunct[8pt]{.}\]

We now move on to showing that the two pushouts exist. Intuitively, the first pushout indicates that $\sk_1 X \star []$ is a free category formed by adding to $X_0$ all the non-identity arrows that exists in $X_1^{\nd}$ and freely composing them. Similarly, the second pushout indicates that $\sk_2 X \star []$ is the category derived from the free category $\sk_1 \star []$ by quotienting the relations induced by the non-degenerate $2$-simplices of $X_2^{\nd}$. In particular, the new hom-sets are formed after quotienting by the equivalence relation generated by $f\circ g \sim h$ for each $2$-cell in $X_2^{\nd}$ together with the relation $afb \sim agb$ for any $a,b$ whenever $f \sim g$.

\begin{lemma}\label{lemma:FXcons}
	The pushout of the following diagram exist:
	
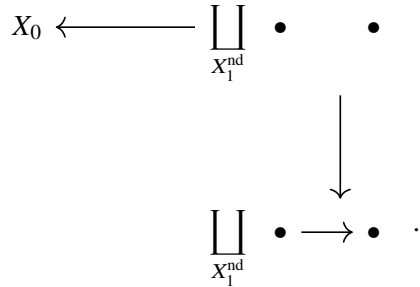
\[\begin{tikzcd}
		{X_0} && {\Coprod\limits_{X^{\nd}_1}\usebox{\twodots}} \\
		\\
		&& {\Coprod\limits_{X^{\nd}_1} \usebox{\anarrow}}
		\arrow[from=1-3, to=1-1]
		\arrow[shift left=5, from=1-3, to=3-3]
	\end{tikzcd}\cdpunct[26pt]{.}\]
\end{lemma}
\begin{proof}
	We first define a category called $FX$, where $F$ is read as free. On the objects, we define it by $\obj(FX) = X_0$. The hom set $FX(x_1,x_2)$ for any $x_1,x_2 \in X_0$ is defined as the collection of all finite unidirectional sequences of $1$-simplices of the form $x_1 \rightarrow \cdots \rightarrow x_2$, with all $1$-simplices belonging to $X_1^{\nd}$. If $x_1=x_2$, we also allow the empty sequence to exist and act as the identity morphism. Composition maps are now defined using concatenation of sequences which are associative and unital. Thus, $FX$ is a well defined category.
	
	We now show that $FX$ satisfies the universal property of the pushout. First, we define the projection maps: The map $\coprod\limits_{X^{\nd}_1}\usebox{\anarrow} \rightarrow FX$ is defined by picking out every \emph{generating} arrow in the category $FX$. That is, we pick all arrows in $FX$ that correspond to a sequence of the form $x_1 \rightarrow x_2$. These are the \emph{generators} as every other arrow is a unique finite composite of the arrows corresponding to $X_1^{\nd}$ by definition of $FX$. And secondly, the map $X_0 \rightarrow FX$ is simply a bijection of objects.
	
	Now, for all $\cD \in \Cat$ that makes the square commute, the map $X_0 \rightarrow D$ will uniquely define a function $\obj(FX) \rightarrow \obj(\cD)$. And on the morphisms, the map $\coprod\limits_{X^{\nd}_1}\usebox{\anarrow} \rightarrow \cD$ uniquely define a \emph{partial function} $\mor(FX) \rightarrow \mor(\cD)$. However, we can further uniquely extend this partial function to an actual function by forcing functorialty in $\mor(FX)$. This whole data provides us the desired unique functor $FX \rightarrow \cD$ that commutes with the projection maps onto $FX$.
	
	In particular, the existence of this pushout implies that the colimit $\sk_1 X \star []$ exists and that the canonical comparison map $FX \rightarrow \sk_1 X \star []$ is an isomorphism. 
\end{proof}

\begin{lemma}\label{lemma:hXcons}
	The pushout of the following diagram exists:
	\[\begin{tikzcd}
		FX && {\Coprod\limits_{X^{\nd}_2} \usebox\freetwosimp} \\
		\\
		&& {\Coprod\limits_{X^{\nd}_2} \usebox\twosimp}
		\arrow[from=1-3, to=1-1]
		\arrow[shift left=5, from=1-3, to=3-3]
	\end{tikzcd}\cdpunct[20pt]{.}\]
\end{lemma}
	
\begin{proof}
	We first define a category called $hX$, where $h$ is read as homotopy. We set $\obj(hX)=X_0$ and the hom sets $hX(x_1,x_2)$ for all $x_1,x_2 \in X_0$ are derived from $FX(x_1,x_2)$ after quotienting each of these hom set with a respective equivalence relation $\sim_{x_1,x_2}$. This equivalence relation is essentially generated by imposing relations corresponding to the non-degenerate $2$-simplices of $X$ that identifies two composable chains. We now provide a more concrete description of this equivalence relation using the notion of chains on $FX$.
	
	We define a \textit{chain} $\sigma$ from $x_1$ to $x_2$ as a map $ \sigma : [n] \rightarrow FX$ such that $\{0\} \mapsto x_1$ and $\{n\}\mapsto x_2$. Note that the image of a chain can contain identity morphisms of $FX$. Let $\Chain(x_1,x_2)$ denote the set of all possible chains from $x_1$ to $x_2$. There is a surjective map $\pi : \Chain(x_1,x_2) \rightarrow FX(x_1,x_2)$ provided by the restriction along the inclusion $\{0 \mapsto n\} \hookrightarrow [n]$. For the surjectivity of this map, note that it has a canonical section. It is defined by sending a map $f \in FX(x_1,x_2)$ to the unique non-degenerate chain $\sigma_{f,\nd}$ from $x_1$ to $x_2$ such that all of the non-trivial morphism of $[n]$ maps to the unique decomposition of $f \in FX$ through the generators. 
	
	We define a relation between two maps $f_1,f_2 \in FX(x_1,x_2)$ by saying $f_1 \sim f_2$ iff they have some representing chain $\sigma_{f_1}$ and $\sigma_{f_2}$ whose images in $FX$ only differ by the boundary of a $2$-simplex in $X_2^{\nd}$. For example, the following diagram represents two such chains:
	
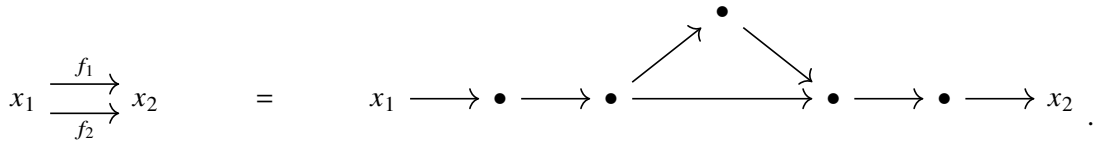
\[\begin{tikzcd}
		&&&&&& \bullet &&& \\
		{x_1} & {x_2} & {=} & {x_1} & \bullet & \bullet && \bullet & \bullet & {x_2}
		\arrow[from=1-7, to=2-8]
		\arrow["{f_1}", shift left=2, from=2-1, to=2-2]
		\arrow["{f_2}"', shift right=2, from=2-1, to=2-2]
		\arrow[from=2-4, to=2-5]
		\arrow[from=2-5, to=2-6]
		\arrow[from=2-6, to=1-7]
		\arrow[from=2-6, to=2-8]
		\arrow[from=2-8, to=2-9]
		\arrow[from=2-9, to=2-10]
	\end{tikzcd}\cdpunct[8pt]{.}\]

	Then, the equivalence relation $\sim_{x_1,x_2}$ on $FX(x_1,x_2)$ is defined as the minimally generated one upon the aforementioned relation. Now we have to show that the composition remain well defined under the quotienting of homsets. In particular, we wish to show that the concatenation operation on $FX$ respects this relation. However, this can be demonstrated pictorially for some $f_1 \sim f_1'$ and $f_2 \sim f_2'$ as follows:
	
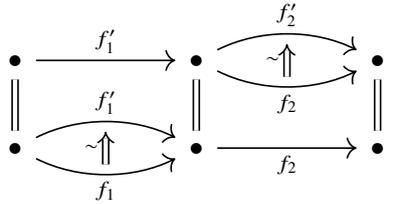
\[\begin{tikzcd}
		\bullet && \bullet && \bullet \\
		\bullet && \bullet && \bullet
		\arrow["{f_1'}", from=1-1, to=1-3]
		\arrow[""{name=0, anchor=center, inner sep=0}, "{f_2}"', curve={height=12pt}, from=1-3, to=1-5]
		\arrow[""{name=1, anchor=center, inner sep=0}, "{f_2'}", curve={height=-12pt}, from=1-3, to=1-5]
		\arrow[equals, from=2-1, to=1-1]
		\arrow[""{name=2, anchor=center, inner sep=0}, "{f_1}"', curve={height=12pt}, from=2-1, to=2-3]
		\arrow[""{name=3, anchor=center, inner sep=0}, "{f_1'}", curve={height=-12pt}, from=2-1, to=2-3]
		\arrow[equals, from=2-3, to=1-3]
		\arrow["{f_2}"', from=2-3, to=2-5]
		\arrow[equals, from=2-5, to=1-5]
		\arrow["\sim", shorten <=3pt, shorten >=3pt, Rightarrow, from=0, to=1]
		\arrow["\sim", shorten <=3pt, shorten >=3pt, Rightarrow, from=2, to=3]
	\end{tikzcd}\cdpunct[20pt]{.}\] 
	In particular, if $f_1 \sim f_1'$ then $f_2 \circ f_1 \sim f_2 \circ f_1'$. The same holds in the $f_2$ arguments after fixing $f_1$. The associativity and unitality of composition in $hX$ is now a direct consequence of the same for $FX$. This proves that $hX$ is a category.

	We now move on to the universal property of $hX$. The functor $FX \rightarrow hX$ is defined on objects as identity and on morphisms through the quotient maps of the equivalence relations $\sim_{x_1,x_2}$. In particular, this implies that each arrow $[f] \in hX(x_1,x_2)$ can be represented using the unique non-degenerate chain $\sigma_{f,\nd}$ from $x_1$ to $x_2$ and that $hX$ is generated using the same \enquote{generating} arrows as $FX$. 
	
	For the map \[\Coprod\limits_{X^{\nd}_2} \usebox\twosimp \rightarrow hX,\] we note that every functor is uniquely specified by its action on the arrows. Consequentially, the map $FX \rightarrow hX$ already defines the action for each $2$-simplex in $X_2^{\nd}$ by simply sending $f,g,h$ to $[f],[g],[h]$. Clearly, $gf \sim h$ and thus the functor is well defined for each $2$-simplex in 
	$X_2^{\nd}$.
	
	Finally, consider any $\cD \in \Cat$ that makes the square commute as follows:
	\[\begin{tikzcd}
		{\Coprod\limits_{X^{\nd}_2} \usebox\freetwosimp} &&& FX \\
		\\
		\\
		{\Coprod\limits_{X^{\nd}_2} \usebox\twosimp} &&& \cD
		\arrow[from=1-1, to=1-4]
		\arrow[shift left=5, from=1-1, to=4-1]
		\arrow["\phi", from=1-4, to=4-4]
		\arrow[from=4-1, to=4-4]
	\end{tikzcd}\cdpunct[32pt]{.}\]
	The map $\phi : FX \rightarrow \cD$ already provides all the desired data on how the functor $\tilde \phi : hX \rightarrow \cD$ ought to be uniquely defined on both the objects and the arrows. However, we do have to check whether it respects the family of equivalence relations for it to be well defined.
	
	The commuting diagram ensures that for all $f,g,h \in FX$ that are the boundary of a non-degenerate $2$-simplex, we have the relation $\phi(h)=\phi(g)\circ\phi(f)$. Note that the images of $f,g$, and $h$ in $FX$ are either the generating arrows of $FX$ or some identity map in $FX$. Consequentially, if there exists two arbitrary $p, q \in FX(x_1,x_2)$ such that $p \sim q$ by the generating relation then $\phi(p)=\phi(q)$ because $p$ and $q$ differ only by a $2$-simplex. Furthermore, if $p \sim q$ by the generated equivalence relation then there is a finite zig-zag of the generating relations between them. In turn, this also implies $\phi(p)=\phi(q)$. This concludes that the functor $\phi$ respects the family of equivalence relations $\sim_{x_1,x_2}$. 
	
	Hence, $\phi : FX \rightarrow \cD$ uniquely descends to a functor $\tilde \phi : hX \rightarrow \cD$ and we are done.
\end{proof}
	
The existence of this pushout implies that the colimit $\sk_2 X \star []$ exists and satisfies $\sk_2 \star [] \cong hX$. Furthermore, the construction $hX$ can be promoted to a functor in $X$ as the $2$-skeletal filtration is functorial in $X$. Hence, we obtain the following main results.

\begin{theorem}
	The weighted colimit $X \star []$ exist and there is a natural isomorphism $X \star [] \cong hX$.
\end{theorem}

\begin{corollary}
	The left adjoint $h$ to the nerve functor $N : \Cat \rightarrow \sSet$ exist and is defined by pushout constructions of \cref{lemma:FXcons} and \cref{lemma:hXcons}. 
\end{corollary}

Finally, we invoke again the standard results about reflective embedding (see \cite[Prop.~4.5.15]{reihlcat}) to provide the following explicit formulas for computing limits and colimits using the $h \dashv N$ adjunction. 

\begin{corollary}
	The category $\Cat$ is complete and cocomplete. Moreover, for each small diagram $F :\cJ \rightarrow \Cat$, we have isomorphisms $\Lim F \cong h (\Lim N \circ F)$ and $\Colim F \cong h (\Colim N \circ F)$.
\end{corollary}
	\section{Related results}

\subsection{Composite adjunctions}\label{sec:deradj}

There are two other important adjunctions that can be derived directly from the $h \dashv N$ adjunction by composing it with the $\Sk_n \dashv (-)_{\geq n}$ adjunction for some $n\geq 0$. The first case at $n=0$ is the usual underlying set adjunction between $\Cat$ and $\Set$ using the $\obj$ functor.

\begin{lemma}
	There exists an adjunction
	\[\begin{tikzcd}
		\Set && \Cat
		\arrow[""{name=0, anchor=center, inner sep=0}, "{h(\Sk_0-)}", curve={height=-12pt}, from=1-1, to=1-3]
		\arrow[""{name=1, anchor=center, inner sep=0}, "{(N-)_{\leq 0}}", curve={height=-12pt}, from=1-3, to=1-1]
		\arrow["\dashv"{anchor=center, rotate=-90}, draw=none, from=0, to=1]
	\end{tikzcd}\cdpunct[0pt]{,}\]
	where the functor $(N-)_{\leq 0}$ is isomorphic to $\obj(-)$ and the left adjoint is isomorphic to the discrete category functor.
\end{lemma}

However, the adjoint pair $\Sk_0 \dashv (-)_{\geq 0}$ rather fits into the following quadruple:
\[\begin{tikzcd}
	\sSet &&& \Set
	\arrow[""{name=0, anchor=center, inner sep=0}, "{(-)_{\leq 0}}"{description}, from=1-1, to=1-4]
	\arrow[""{name=1, anchor=center, inner sep=0}, "{\pi_0(-) \cong \coeq(d_0,d_1 : (-)_1 \rightarrow (-)_0)}", shift left=5, curve={height=-30pt}, from=1-1, to=1-4]
	\arrow[""{name=2, anchor=center, inner sep=0}, "{\Sk_0}"{description}, shift right=2, curve={height=18pt}, from=1-4, to=1-1]
	\arrow[""{name=3, anchor=center, inner sep=0}, "{\Cosk_0}", shift left=2, curve={height=-18pt}, from=1-4, to=1-1]
	\arrow["\dashv"{anchor=center, rotate=-91}, draw=none, from=0, to=3]
	\arrow["\dashv"{anchor=center, rotate=-88}, draw=none, from=1, to=2]
	\arrow["\dashv"{anchor=center, rotate=-89}, draw=none, from=2, to=0]
\end{tikzcd}\cdpunct[0pt]{.}\]

Therefore, it is natural to ask whether these additional adjoints $\Cosk_0$ and $\pi_0$ descend to the level of categories. Fortunately, both the functors $\Sk_0,\Cosk_0$ factors through the fully faithful functor $N : \Cat \rightarrow \sSet$. One readily verifies this by checking that for all $X \in \Set$, both $\Sk_0(X)$ and $\Cosk_0(X)$ satisfy the spine extension property. Do recall that $\Cosk_0(X)$ is a simplicial set with $\Cosk_0(X)_n= X^{n+1}$, such that the face maps $d_i : X^{n+1} \rightarrow X^n$ are the product projection maps, and the degeneracy maps $s_i : X^{n} \rightarrow X^{n+1}$ are the diagonal map at the $i$th factor, i.e., $x_i \mapsto (x_i,x_i)$.

\begin{lemma}\label{lemma:catset0}
	There exists the quadruple of adjoints
	\[\begin{tikzcd}
		\Cat &&& \Set
		\arrow[""{name=0, anchor=center, inner sep=0}, "\obj"{description}, from=1-1, to=1-4]
		\arrow[""{name=1, anchor=center, inner sep=0}, "{\pi_0}", shift left=5, curve={height=-30pt}, from=1-1, to=1-4]
		\arrow[""{name=2, anchor=center, inner sep=0}, "{\mathrm{D}}"{description}, shift right=2, curve={height=18pt}, from=1-4, to=1-1]
		\arrow[""{name=3, anchor=center, inner sep=0}, "{\mathrm{ID}}", shift left=2, curve={height=-18pt}, from=1-4, to=1-1]
		\arrow["\dashv"{anchor=center, rotate=-90}, draw=none, from=0, to=3]
		\arrow["\dashv"{anchor=center, rotate=-89}, draw=none, from=1, to=2]
		\arrow["\dashv"{anchor=center, rotate=-90}, draw=none, from=2, to=0]
	\end{tikzcd}\cdpunct{.}\]
	Where, we call $h\circ\Sk_0$ as the functor $\mathrm{D}$ for discrete category, $h\circ\Cosk_0$ as the functor $\mathrm{ID}$ for indiscrete category, $(N-)_{\leq 0}$ as the functor $\obj$, and finally $\pi_0(N-)$ as simply $\pi_0$. Furthermore, this quadruple factors through the fully faithful functor $i : \Grpd \hookrightarrow \Cat$. 
\end{lemma}
\begin{proof}
	For the final statement regarding factoring through $i : \Grpd \hookrightarrow \Cat$, simply observe that both $\mathrm{D}$ and $\mathrm{ID}$ are discrete and indiscrete groupoids respectively. In particular, $\mathrm{ID}$ is a category with the property that for all objects $c,d \in \mathrm{ID}(X)$, the hom set satisfies $|\mathrm{ID}(X)(c,d)|=1$.
\end{proof}

The second adjunction that we obtain is at the case of $n=1$. It the usual free category adjunction involving $\Cat$ and \emph{reflexive quivers}. Note that we are identifying the category of reflexive quivers with the isomorphic choice $\sSet^{\leq 1}$.

\begin{lemma}\label{lemma:catset1}
	There exists an adjunction
	\[\begin{tikzcd}
		{\sSet^{\leq 1}} && \Cat
		\arrow[""{name=0, anchor=center, inner sep=0}, "{h(\Sk_1-)}", curve={height=-12pt}, from=1-1, to=1-3]
		\arrow[""{name=1, anchor=center, inner sep=0}, "{(N-)_{\leq 1}}", curve={height=-12pt}, from=1-3, to=1-1]
		\arrow["\dashv"{anchor=center, rotate=-90}, draw=none, from=0, to=1]
	\end{tikzcd}\cdpunct{,}\]
	where $h(\Sk_1-)$ is the free category functor on a reflexive quiver, and $(N-)_{\leq 1}$ is the underlying reflexive quiver functor. However, there is neither any factoring of this adjunction through $\Grpd$ nor any further left or right adjoints. 
\end{lemma}

There is still another important category to consider in the $n=1$ case, which is $\Set^{[0]\rightrightarrows [1]}$. We denote this category as $\sSet^{\leq 1}_s$. The key point is that $\sSet^{\leq 1}_s$ is also equivalently the category of quivers. This yields the composite adjoints
\[\begin{tikzcd}
	{\sSet^{\leq1}_s} && {\sSet^{\leq 1}} && \Cat
	\arrow[""{name=0, anchor=center, inner sep=0}, "{\coprod_s}", shift left=2, from=1-1, to=1-3]
	\arrow["{F \circ \coprod_{s}}", shift left=4, curve={height=-30pt}, from=1-1, to=1-5]
	\arrow[""{name=1, anchor=center, inner sep=0}, "{U_{s}}", shift left=2, from=1-3, to=1-1]
	\arrow[""{name=2, anchor=center, inner sep=0}, "F", shift left=2, from=1-3, to=1-5]
	\arrow["{U_s \circ U}", shift left=4, curve={height=-30pt}, from=1-5, to=1-1]
	\arrow[""{name=3, anchor=center, inner sep=0}, "U", shift left=2, from=1-5, to=1-3]
	\arrow["\dashv"{anchor=center, rotate=-90}, draw=none, from=0, to=1]
	\arrow["\dashv"{anchor=center, rotate=-90}, draw=none, from=2, to=3]
\end{tikzcd}\cdpunct{,}\]
exhibiting the usual free category and underlying quiver construction. Here, $\coprod_{s}$ is the functor that freely adjoins degeneracies and the functor $U_{s}$ forgets the existence of the degeneracy map $s_0 : X_0 \rightarrow X_1$.

Finally, any consideration of the $n \geq 2$ case is redundant as $h \circ \sk_2 \cong h$ using the counit map $\epsilon^{\Sk_2}$. This is further justified by noting that both the left adjoints $\mathrm{D}$ and $F$ are precisely the auxiliary steps that appear in the pushout construction of $h$ which terminates at the level $2$. 

\subsection{\tpdf{Coequalizers in $\Cat$}}

While we have already provided a complete description of coproducts, we haven't yet done the same for coequalizers. Indeed, we intentionally chose to avoid them until now because of their complexity and their nonessential role in our proof of cocompleteness. To further justify this choice, we now derive a complete description of them using the reflecting embedding $N : \Cat \rightarrow \sSet$.

Consider any two parallel arrow diagram $F,G: \cC \rightrightarrows \cD$ in $\Cat$. The colimit of this diagram is defined as the coequalizer $\coeq(F,G)$. Using the nerve functor, we can reduce the problem to computing the coequalizer of the diagram $NF,NG : N\cC \rightrightarrows N\cD$ (i.e., $\coeq(NF,NG)$) and then computing its homotopy category. Although coequalizers in simplicial sets are not trivial either, they can still be computed levelwise as coequalizers in $\Set$, whose description is well known. Moreover, we only require this computation up to the $2$-simplices level as this data is sufficient to compute the homotopy category.

Our first goal is to compute the free category $F\coeq(NF,NG)$, which only depends on the two sets: $\coeq(NF_0,NG_0)$ and $\coeq(NF_1,NG_1)$. Let $\pi_0: N\cD_0 \rightarrow \coeq(NF_0,NG_0)$ be the canonical coequalizer map whose action on the elements is denoted as $x \mapsto [x]$. The key idea is that $\coeq(NF_1,NG_1)$ can be computed locally through a collection of local coequalizers for all $[x],[y] \in \coeq(NF_0,NG_0)$.

\begin{definition}
	For a given $X \in \sSet^{\leq 1}$ and all $a,b \in X_0$, define an indexed family of sets \[\{X_1(a,b)\}_{(a,b) \in X_0 \times X_0} \in \Set\] by partitioning the set $X_1$ into the equivalences classes defined by the source and target pairs $(a,b)$. These sets are called the \emph{local (1-simplices) sets}.
\end{definition} 
\begin{lemma}
	 The local coequalizer $coeq(NF_1,NG_1)([x],[y])$ is computed by the following: \[\coeq(NF_1,NG_1)([x],[y]) \cong \left(\Coprod\limits_{(i,j)\in [x] \times [y]} N\cD_1(i,j)\right) \Big/ (NF_1 \sim NG_1) \cdpunct{.}\]
\end{lemma}
\begin{proof}
	Both $NF$ and $NG$ are simplicial maps. Hence, they are compatible with the source, target, and identity maps. Consequentially, the coequalizer relations $NF_{\{i\}}(\sigma) \sim NG_{\{i\}}(\sigma)$ for all $\sigma \in N\cC_{\{i\}}$ and $i \in \{0,1\}$ are also stable under these three maps. As a result, the coequalizer equivalence relation generated by $NF_1 \sim NG_1$ is also compatible with $NF_0 \sim NG_0$ using the source, target, and identity maps. This provides the desired expression for the local coequalizer.
\end{proof}

The free category $F\coeq(NF,NG)$ is then the category formed by freely composing the arrows in the $1$-skeletal simplicial set $\coeq(NF,NG)_{\leq 1}$. Afterwards, we only need to identify the compositions corresponding to whenever there exist some $2$-simplex in $\coeq(NF_2,GF_2)$ filing a particular $2$-sphere. However, the canonical coequalizer map $N\cD_2 \rightarrow \coeq(NF_2,GF_2)$ is a surjective map as it a coequalizer of sets. So, we only require existence the existence of some $2$-simplex in the preimage of this map.

\begin{remark}
	In our construction of $hX$, the generators are given by declaring $gf \sim h$ whenever there exists a $2$-simplex filling the $2$-sphere formed by $f,g,h$. In particular, this relation does not depend upon which $2$-simplex it is or how many such $2$-simplices there are. Hence, we only need to choose one $2$-simplex filling a given $2$-sphere.
\end{remark}

Using this observation together with the description of $F\coeq(NF,NG)$, we obtain the following lemma providing a complete description of coequalizers in $\Cat$.

\begin{lemma} \label{lemma:coeqcat}
	The coequalizer $\coeq(F,G)$ is described as follows:
	\begin{description}
		\item[objects] are given by the collection $\coeq(NF_0,NG_0)$,
		\item[morphisms] are given by the collection of homsets $F\coeq(NF,NG)([x],[y])$ with objects $[x],[y] \in \coeq(NF_0,NG_0)$, subject to the relations $[g]\circ[f]=[h]$ whenever there exists any corresponding representatives $f',g',h' \in N\cD_1$ that are the boundary of a $2$-simplex in $N\cD_2$.
	\end{description}
\end{lemma}

The important takeaway is that verifying the universal property of above construction is highly intricate and cumbersome. In particular, our approach was designed to circumvent this difficulty. 

\subsection{Localization}\label{sec:loc}

Let $I$ be the free isomorphism category and let $F:\cC \rightarrow \cD$ be an arbitrary functor that satisfy $F(f) \in \iso(\cD)$ for all $f \in \cC$.  Each commuting square 
\[\begin{tikzcd}
	{\Coprod\limits_{f \in \mor \cC} [1]} && \cC \\
	\\
	{\Coprod\limits_{f \in \mor \cC} I} && \cD
	\arrow[from=1-1, to=1-3]
	\arrow[hook', from=1-1, to=3-1]
	\arrow["F", from=1-3, to=3-3]
	\arrow[from=3-1, to=3-3]
\end{tikzcd}\]
encapsulate such a functor $F: \cC \rightarrow \cD$, which factors as $F: \cC \rightarrow \iso(\cD) \hookrightarrow \cD$ by definition. Conversely, each such functor that factors through the isomorphisms extends uniquely to such a square.

Thus, we can construct and analyze a universal such square by using pushouts. This pushout allows us to construct a non-trivial adjoint to the inclusion functor $\Grpd \hookrightarrow \Cat$.

\begin{definition}
	Let $\cL(\cC)$ be the pushout defined by the following square:
	\[\begin{tikzcd}
		{\Coprod\limits_{f \in \mor \cC} [1]} && \cC \\
		\\
		{\Coprod\limits_{f \in \mor \cC} I} && {\cL(\cC)}
		\arrow[from=1-1, to=1-3]
		\arrow[hook', from=1-1, to=3-1]
		\arrow["{\gamma_{\cC}}", from=1-3, to=3-3]
		\arrow[from=3-1, to=3-3]
		\arrow["\ulcorner"{anchor=center, pos=0.125, rotate=180}, draw=none, from=3-3, to=1-1]
	\end{tikzcd}\cdpunct[20pt]{.}\]
	We call $\cL(\cC)$ the \emph{total localization} of $\cC$ and $\gamma_{\cC}$ as the \emph{localization map} of $\cC$.
\end{definition}

While computing $\cL(\cC)$ is difficult, we can utilize the nerve functor to provide a concrete description for it. To this end, we define $\cL_{N}(N\cC)$ to be the pushout given by the following diagram:
\[\begin{tikzcd}
	{\Coprod\limits_{f \in \mor \cC} N[1]} && {N\cC} \\
	\\
	{\Coprod\limits_{f \in \mor \cC} NI} && {\cL_{N} (N\cC)}
	\arrow[from=1-1, to=1-3]
	\arrow[hook', from=1-1, to=3-1]
	\arrow[from=1-3, to=3-3]
	\arrow[from=3-1, to=3-3]
	\arrow["\ulcorner"{anchor=center, pos=0.125, rotate=180}, draw=none, from=3-3, to=1-1]
\end{tikzcd}\cdpunct[20pt]{.}\]
Now using that $N$ preserve arbitrary coproducts, we obtain $\cL(\cC) \cong h \cL_{N}(N\cC)$. Importantly, the computation of $\cL_{N}$ is much simpler. As a result, we obtain the following description of $\cL(\cC)$.

\begin{lemma}
	The map $\gamma_{\cC} : \cC \rightarrow \cL(\cC)$ is bijective on objects. On morphisms, any element in $\cL(\cC)(x,y)$ can be represented by a (non-unique) finite zig-zag in $\cC$ of the form
	
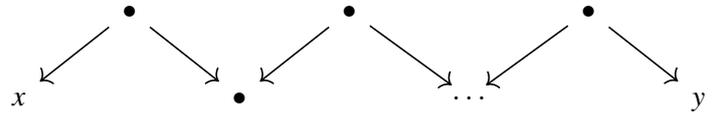
\[\begin{tikzcd}
		& \bullet && \bullet && \bullet \\
		x && \bullet && \cdots && y
		\arrow[from=1-2, to=2-1]
		\arrow[from=1-2, to=2-3]
		\arrow[from=1-4, to=2-3]
		\arrow[from=1-4, to=2-5]
		\arrow[from=1-6, to=2-5]
		\arrow[from=1-6, to=2-7]
	\end{tikzcd}\cdpunct[4pt]{,}\]
	subject to the following constraints for all $f,g \in \mor \cC$:
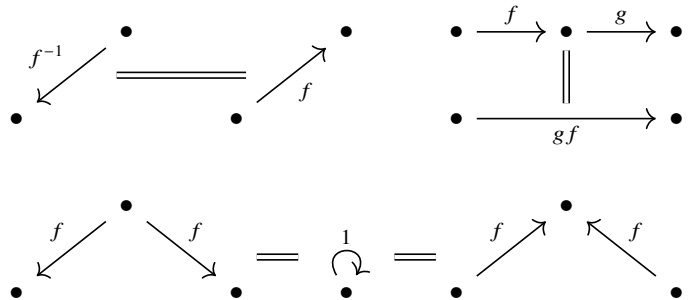
	\[\begin{tikzcd}
		& \bullet && \bullet & \bullet & \bullet & \bullet \\
		\bullet && \bullet && \bullet && \bullet \\
		& \bullet &&&& \bullet \\
		\bullet && \bullet & \bullet & \bullet && \bullet
		\arrow[""{name=0, anchor=center, inner sep=0}, "{f^{-1}}"', from=1-2, to=2-1]
		\arrow["f", from=1-5, to=1-6]
		\arrow["g", from=1-6, to=1-7]
		\arrow[""{name=1, anchor=center, inner sep=0}, "f"', from=2-3, to=1-4]
		\arrow[""{name=2, anchor=center, inner sep=0}, "gf"', from=2-5, to=2-7]
		\arrow["f"', from=3-2, to=4-1]
		\arrow["f", from=3-2, to=4-3]
		\arrow[shift left=5, between={0}{0.6}, equals, from=4-3, to=4-4]
		\arrow["1", from=4-4, to=4-4, loop, in=60, out=120, distance=5mm]
		\arrow["f", from=4-5, to=3-6]
		\arrow[shift right=5, between={0}{0.6}, equals, from=4-5, to=4-4]
		\arrow["f"', from=4-7, to=3-6]
		\arrow[between={0.2}{0.8}, equals, from=0, to=1]
		\arrow[between={0}{0.8}, equals, from=1-6, to=2]
	\end{tikzcd}\cdpunct[6pt]{.}\]
	The map $\gamma_{\cC} : \cC \rightarrow \cL(\cC)$ is then given by sending $f: x \rightarrow y \in \cC$ to the morphism in $\cL(\cC)(x,y)$ represented by the chain $\{ x \rightarrow y\}$.
\end{lemma}

\begin{proof}
	The bijection on objects of the functor $\gamma_{\cC} : \cC \rightarrow \cL(\cC)$ follows directly from the construction of pushouts. In particular, we are not modifying the $0$-simplices by the constructing of the pushout $\cL_{N}$ or by applying the functor $h$.
	
	At the level of morphisms, we are freely adding an inverse arrow corresponding to each $f \in \mor \cC$. We can label these newly added arrows as well as the already existing isomorphisms as $f^{-1}$ for each such $f$. Note that if a map $f : x \rightarrow y \in \cC$ already has an inverse $f^{-1} : y \rightarrow x \in \cC$ then the uniqueness of inverses ensures that whatever new inverse we freely add for $f$ is already identified with $f^{-1}$. As a consequence, at the free category level, the homsets will contain strings of maps generated by $f$ and $f^{-1}$. For example, $f_1f_2^{-1}f_1f_3^{-1}$ is one such possible string. At the level of homotopy category, we have to enforce the relations imposed by the original category $\cC$, as well as impose that $ff^{-1}=1=f^{-1}f$ for each $f$. This provides the desired identification with a zigzag by the depicting the maps $f^{-1}$ in the string as a backward arrow in the zigzag.
\end{proof}

The localization construction $\cL(-)$ is functorial in $\Cat$. In particular, it defines a functor $\cL: \Cat \rightarrow \Cat$ or $\Cat \rightarrow \Grpd$. Moreover, it is a left adjoint to the embedding $i : \Grpd \hookrightarrow \Cat$. These properties follow fairly straightforwardly from the universal property of the pushouts. Furthermore, the canonical map $\gamma_{\cC}$ also extends to a natural transformation $\gamma : 1_{\Cat} \Rightarrow \cL$ by the universality of pushouts. 

\begin{lemma}
	There exists a triple of adjoints
	\[\begin{tikzcd}
		\Grpd &&&& \Cat
		\arrow[""{name=0, anchor=center, inner sep=0}, "i"{description}, hook, from=1-1, to=1-5]
		\arrow[""{name=1, anchor=center, inner sep=0}, "\iso", shift left=3, curve={height=-18pt}, from=1-5, to=1-1]
		\arrow[""{name=2, anchor=center, inner sep=0}, "\cL"', shift right=3, curve={height=18pt}, from=1-5, to=1-1]
		\arrow["\dashv"{anchor=center, rotate=-91}, draw=none, from=0, to=1]
		\arrow["\dashv"{anchor=center, rotate=-89}, draw=none, from=2, to=0]
	\end{tikzcd}\cdpunct[0pt]{,}\]
	such that $\gamma$ is the unit of the adjunction $\cL \dashv i$.
\end{lemma}

\begin{definition}
	We call the functor $\cL : \Cat \rightarrow \Grpd$ (or $\Cat$) as the \emph{total localization} functor. 
\end{definition}

Total localization functor is obtained by freely inverting all the morphisms in a category. However, we might want to freely invert only a certain class of morphisms. We first define the category of Marked Categories to precisely encode our object of study.

\begin{definition}
	We call a pair $(\cC,U_{\cC})$ with $\cC \in \Cat$ and $U_{\cC} \subseteq \mor \cC$ as a \emph{marked category}.
\end{definition}

\begin{definition}
	There exists a category $\MarCat$ whose objects are marked categories and whose morphisms $F : (\cC, U_{\cC}) \rightarrow (\cD,U_{\cD})$ are defined by functors $F: \cC \rightarrow \cD$ that preserves the marking, that is, $F(U_{\cC}) \subseteq U_{\cD}$. $\MarCat$ has same compositions and identities as induced from $\Cat$.
\end{definition}

Now we have the following analogous constructions for the relative localization denoted as $\cL(\cC,U_{\cC})$.

\begin{definition}
	For each marked category $(\cC,U_{\cC})$, let $\cL(\cC, U_{\cC})$ denote the pushout given by the following pushout square:
	\[\begin{tikzcd}
		{\Coprod\limits_{f \in U_{\cC}} [1]} && \cC \\
		\\
		{\Coprod\limits_{f \in U_{\cC}} I} && {\cL(\cC,U_{\cC})}
		\arrow[from=1-1, to=1-3]
		\arrow[hook', from=1-1, to=3-1]
		\arrow["{\gamma_{\cC,U_{\cC}}}"{pos=0.7}, from=1-3, to=3-3]
		\arrow[from=3-1, to=3-3]
		\arrow["\ulcorner"{anchor=center, pos=0.125, rotate=180}, draw=none, from=3-3, to=1-1]
	\end{tikzcd}\cdpunct[20pt]{.}\]
	The object $\cL(\cC,U_{\cC})$ is called \emph{relative localization} of $\cC$ with respect to $U_{\cC}$. The map $\gamma_{\cC,U_{\cC}}$ is called the \emph{localization map}.
	
\end{definition}

\begin{lemma}
	The map $\gamma_{\cC,U_{\cC}} : \cC \rightarrow \cL(\cC,U_{\cC})$ is bijective on objects. On morphisms, any element in $\cL(\cC,U_{\cC})(x,y)$ can be represented by a (non-unique) finite zig-zag in $\cC$ of the form
	
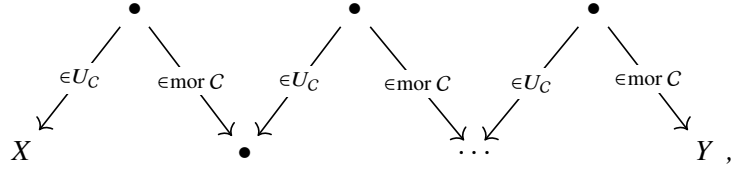
\[\begin{tikzcd}
		& \bullet && \bullet && \bullet \\
		\\
		X && \bullet && \cdots && Y
		\arrow["{\in U_{\cC}}"{description}, from=1-2, to=3-1]
		\arrow["{\in \mor \cC}"{description}, from=1-2, to=3-3]
		\arrow["{\in U_{\cC}}"{description}, from=1-4, to=3-3]
		\arrow["{\in \mor \cC}"{description}, from=1-4, to=3-5]
		\arrow["{\in U_{\cC}}"{description}, from=1-6, to=3-5]
		\arrow["{\in \mor \cC}"{description}, from=1-6, to=3-7]
	\end{tikzcd}\cdpunct[4pt]{,}\]
	subject to the following constraints for all $w \in U_{\cC}$ and $f,g \in \mor \cC$:
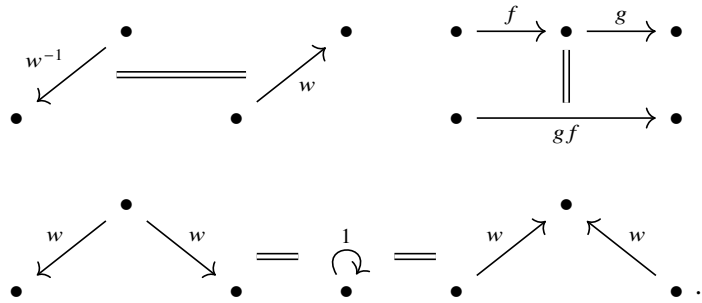
	\[\begin{tikzcd}
		& \bullet && \bullet & \bullet & \bullet & \bullet \\
		\bullet && \bullet && \bullet && \bullet \\
		& \bullet &&&& \bullet \\
		\bullet && \bullet & \bullet & \bullet && \bullet
		\arrow[""{name=0, anchor=center, inner sep=0}, "{w^{-1}}"', from=1-2, to=2-1]
		\arrow["f", from=1-5, to=1-6]
		\arrow["g", from=1-6, to=1-7]
		\arrow[""{name=1, anchor=center, inner sep=0}, "w"', from=2-3, to=1-4]
		\arrow[""{name=2, anchor=center, inner sep=0}, "gf"', from=2-5, to=2-7]
		\arrow["w"', from=3-2, to=4-1]
		\arrow["w", from=3-2, to=4-3]
		\arrow[shift left=5, between={0}{0.6}, equals, from=4-3, to=4-4]
		\arrow["1", from=4-4, to=4-4, loop, in=60, out=120, distance=5mm]
		\arrow["w", from=4-5, to=3-6]
		\arrow[shift right=5, between={0}{0.6}, equals, from=4-5, to=4-4]
		\arrow["w"', from=4-7, to=3-6]
		\arrow[between={0.2}{0.8}, equals, from=0, to=1]
		\arrow[between={0}{0.8}, equals, from=1-6, to=2]
	\end{tikzcd}\cdpunct[6pt]{.}\]
	The map $\gamma_{\cC,U_{\cC}} : \cC \rightarrow \cL(\cC,U_{\cC})$ is then given by sending $f: x \rightarrow y \in \cC$ to the morphism in $\cL(\cC,U_{\cC})(x,y)$ represented by the chain $\{ x \rightarrow y\}$.
\end{lemma} 

The construction of relative localization analogously extends to an adjunction.

\begin{lemma}
	Let $\iso_{m}: \Cat \rightarrow \MarCat$ be the marking functor that acts on the objects by $\cC \mapsto (\cC,\iso(\cC))$ and acts as the identity on morphisms. There exists an adjunction
	\[\begin{tikzcd}
		\MarCat &&&& \Cat
		\arrow[""{name=0, anchor=center, inner sep=0}, "\cL", curve={height=-12pt}, from=1-1, to=1-5]
		\arrow[""{name=1, anchor=center, inner sep=0}, "{\iso_m}", curve={height=-12pt}, from=1-5, to=1-1]
		\arrow["\dashv"{anchor=center, rotate=-90}, draw=none, from=0, to=1]
	\end{tikzcd}\cdpunct[0pt]{,}\]
	such the unit maps $\gamma_{\cC,U_{\cC}} : (\cC,U_{\cC}) \rightarrow \iso_m(\cL(\cC,U_{\cC}))$ are canonically defined through the relative localization maps $\gamma_{\cC,U_{\cC}}$ on the underlying categories.
\end{lemma}

	\newpage 

	\clearpage 
	
	\bibliography{master} 
	
\end{document}